\documentclass[10pt]{article}
\usepackage{amssymb,amsmath,amsfonts,amsthm,enumerate,color}
\usepackage[toc,page,title,titletoc,header]{appendix}
\usepackage{graphicx}
\usepackage{indentfirst}
\usepackage{multicol}
\usepackage{booktabs}

\numberwithin{equation}{section}

\newtheorem{Theorem}{Theorem}[section]

\newtheorem{Definition}[Theorem]{Definition}

\newtheorem{Remark}[Theorem]{Remark}
\numberwithin{equation}{section}

\def \Vh0{\stackrel{\circ}{V}_h} \def\to{\rightarrow}

  \def\f{\frac}  

  \def\ss{\smallskip}

\newcommand{\lc}
{\mathrel{\raise2pt\hbox{${\mathop<\limits_{\raise1pt\hbox
{\mbox{$\sim$}}}}$}}}

\newcommand{\gc}
{\mathrel{\raise2pt\hbox{${\mathop>\limits_{\raise1pt\hbox{\mbox{$\sim$}}}}$}}}

\newcommand{\ec}
{\mathrel{\raise2pt\hbox{${\mathop=\limits_{\raise1pt\hbox{\mbox{$\sim$}}}}$}}}

\def\bb{\begin{equation}} \def\ee{\end{equation}}

\def\beqn{\begin{eqnarray}}  \def\eqn{\end{eqnarray}}

\def\beqnx{\begin{eqnarray*}} \def\eqnx{\end{eqnarray*}}

\def\bn{\begin{enumerate}} \def\en{\end{enumerate}}

\def\bd{\begin{description}} \def\ed{\end{description}}

\makeatletter

\newenvironment{figurehere}
  {\def\@captype{figure}}
  {}
\makeatother

\begin{document}

\author{
Habib Ammari\thanks{Department of Mathematics,
ETH Z\"urich,
R\"amistrasse 101, CH-8092 Z\"urich, Switzerland (\tt habib.ammari@math.ethz.ch). }
\and
Yat Tin Chow\thanks{Department of Mathematics, University of California, Riverside, USA ({\tt yattinc@ucr.edu)}.}
\and
Keji Liu\thanks{The corresponding author, shanghai Key Laboratory of Financial Information Technology,  Institute of Scientific Computation and Financial Data Analysis, School of Mathematics, Shanghai University of Finance and Economics, 777 Guoding Road, Shanghai 200433, P.R. China
 ({\tt liu.keji@sufe.edu.cn; \;kjliu.ip@gmail.com}).}
}

\title{\bf Optimal Mesh Size for Inverse Medium Scattering Problems}

\date{}
\maketitle
\begin{abstract}
An optimal mesh size of the sampling region can help to reduce computational burden in practical applications. In this work, we investigate optimal choices of mesh sizes for the identifications of medium obstacles from either the far-field or near-field data in two and three dimensions. The results would have applications in the reconstruction process of inverse scattering problems. 
\end{abstract}

\ss

\ss

\section{Introduction}
The direct and inverse medium scattering problems have been investigated extensively in the past few decades due to their important practical applications in geophysics, nondestructive testing, biological studies, evaluation and medicine, see, for instance, \cite{MSR, medical}. 
A large variety of numerical reconstruction methods have been developed and applied in practice for different kinds of models. For instance, some effective methods are developed for the homogeneous background model: the linear sampling or probing methods (LSM) \cite{CK}, the time-reversal multiple signal classification method (MUSIC) \cite{iakov1,iakov2,music}, the direct sampling method (DSM) \cite{Jin12}, the multilevel sampling method (MSM) \cite{LZ}, the contrast source inversion method (CSIM) \cite{BerBro}, etc. Moreover, several efficient algorithms are applied in various practical situations, e.g., in the stratified ocean waveguide model: the simple method (SM) \cite{L1}, extended direct sampling methods (EDSMs) \cite{CILZ, DSM11, DSM12, DSM13, DSM14, DSM15, L2, LXZ2, LXZ3},  the extended MUSIC method (EMUSIC) \cite{AIK, extended1},  the topological based imaging  functional \cite{topological}, etc. Furthermore, several efficient approaches are introduced in the two-layered medium model: the extended multilevel sampling method (EMSM) \cite{YL}, the extended linear sampling method (ELSM) \cite{DLLY}, etc.

Most reconstruction methods involve an indicator function (called also imaging functional) to be evaluated at each sampling point inside a sampling region for the recovery of the shape (and, if possible, the parameters therein) of the unknown inclusions. 
In practice, the sampling region is usually chosen to be large enough to contain all the unknown inclusions; then a considerably fine mesh is used in the sampling region to reconstruct the details of the objects. However, an extremely fine mesh would drastically increase the computational complexity in practice, especially in the three-dimensional case.  In fact, the computation complexity of most reconstruction methods (see, for instance, DSM, MUSIC, LSM, CSIM, etc.) is of $\mathcal{O}(N^2)$ in  two dimensions and $\mathcal{O}(N^3)$ in three dimensions, where $N$ is the number of sampling points along one direction.  Reconstruction methods such as MSM and the EMSM  has complexity of order of $N \log N$ in two dimensions and $N^2\log N$ in  three dimensions. Hence, they are still quite demanding if $N$ is large.
Therefore, an optimal mesh size for the sampling and the reconstruction is crucial to reduce the computation complexity of these algorithms and at the same time to reconstruct details of the inclusions in practical applications. 
Moreover, several iterative methods (i.e., EMSM, CSIM, etc.) apply certain cut-off values as a stopping criterion of the method.  Such a choice of stopping criterion may heavily depend on the reconstruction quality under a coarser mesh and a subjective choice of the cutoff value; which sometimes causes the reconstruction algorithm to entirely miss some small inclusions.  Choosing an optimal mesh size can improve the quality of the reconstruction  and practically avoid the aforementioned issue. 

To the best of the authors' knowledge, although there are quite a lot of discussions of resolution analysis (e.g., Abbe-Rayleigh resolution limit, Sparrow's resolution limit, etc) in inverse scattering problems, e.g., \cite{resol1, resol2, FWA, GP}, a systematic choice of the optimal mesh size for numerical reconstruction algorithms in inverse scattering problems has not yet been addressed in the literature.  
We would like to remark that the resolution analysis aims at understanding the quality of an image given a fixed grid size; whereas here we aim at choosing an optimal mesh size to ensure certain quality of the reconstruction.
This paper is hence devoted to an optimal selection of the mesh size in reconstruction methods. It is worth emphasizing that, in this paper, instead of focusing on one particular reconstruction method, we are interested in an optimal choice of the mesh size in a numerical reconstruction scheme assuming a maximal resolving power.  We shall provide an optimal selection of mesh size for the identifications of unknown inclusions from both the far-field and the near-field data in two and three dimensions.

This paper is organized as follows. In section \ref{sec:1}, we state the forward problems that we focus on.  We then provide in section \ref{sec:2} an  optimal mesh size selection given far-field data; whereas in section \ref{sec:3}, we consider the mesh size problem when near-field data are available.  
In section \ref{sec:refine}, we describe how our optimal mesh size function may help to refine our mesh in order to resolve clustered inclusions. 
In section \ref{sec:numerics}, we show that our choice of mesh size is indeed optimal by comparing several reconstructions from a typical numerical method with different mesh sizes.
Finally, some concluding remarks are presented in section \ref{con}. The results of this paper can help minimizing the computation complexity of methods designed for solving inverse scattering problems.  A similar analysis can be performed for electromagnetic and elastic inverse wave scattering problems.

\section{Problem description} \label{sec:1}
In this section, we shall describe the direct and inverse scattering problems in the presence of inhomogeneous inclusions (inside a sampling domain $\Omega$; see Figure \ref{fig:demo0}) embedded in a homogeneous background medium $\mathbb{R}^M (M \geq 2)$.   Consider an inhomogeneous inclusion represented by a (variable) contrast $q \in L^{\infty} ( \mathbb{R}^M ) $ such that $\text{supp}(q) \subset \Omega$, i.e., it vanishes outside $\Omega$.

\subsection{Forward problem}

We are now ready to state the considered forward problems.  For numerical simulations of solutions of these forward problems, we refer the reader, for instance, to \cite{LZ, LXZ3}. 

\subsubsection{Plane wave illumination} \noindent Suppose that $u^i$ is an incident plane given by
\beqnx
u^i(x,d_y) := e^ {  i k  \, d_y \cdot x } \,,
\eqnx
where $d_y \in \mathbb{S}^{M}$ is the direction of incidence with $k$ being the wave number and $\mathbb{S}^{M}$ being the unit sphere in $\mathbb{R}^{M}$. Then, $u^i$ satisfies the Helmholtz equation
\beqnx
\Delta u^i+k^2 u^i = 0\quad \text{for}\quad x\in\mathbb{R}^M \,.
\eqnx
We now consider the total field $u$ as the solution to the following equation,
\beqn
\Delta u+k^2\Big(1+q(x)\Big)u = 0\quad \text{for}\quad x\in\mathbb{R}^M \,,
\label{scattering1}
\eqn
subject to the outgoing Sommerfeld radiation condition:
\beqn
   \f{\partial u^s}{\partial r} - i k u^s = o\Big(\frac{1}{r^{(M-1)/2}}\Big) \quad \text{as}\quad r\to\infty,
    \label{sommerfield}
\eqn
which holds uniformly in all directions $x/|x|$, where $r=|x|$ and  $u^s := u - u^i$ is the scattered field.

From \eqref{scattering1}, due to the presence of inclusions in $\Omega$, the Lippmann-Schwinger representation formula of the total field $u$ yields
\beqn
u^s(x, d_y)= u(x,d_y)- u^i(x,d_y) = k^2\int_\Omega q(z)G(x,z)u(z,d_y) \, dz,\quad d_y\in\mathbb{S}^{M}\,,
\label{compute1}
\eqn
where $G(x,z) , x,z \in \,\mathbb{R}^M \text{ and } x \neq z,$ is the Green function associated with the homogeneous background,
\beqn
G(x,z)= C_{M} k^{\frac{2-M}{2}} |x-z|^{\frac{2-M}{2}} H^{(1)}_{\frac{M-2}{2}}(k |x-z|) \,.
\eqn
Here,  $H_\nu^{(1)}$ is the Hankel function  of the  first kind of order $\nu$ and $C_M = -i (2 \pi)^{\frac{2-M}{2}} /4$. In particular, we have
\beqn
G(x,z)= \left\{
\begin{aligned}
&\frac{i}{4}H^{(1)}_0(k|x-z|) &\text{for}\quad M=2,\\[2mm]
&\frac{e^{ik|x-z|}}{4\pi |x-z|} &\text{for}\quad M=3.
\end{aligned}
\right.
\eqn


\begin{figurehere}
 \hfill{}\includegraphics[clip,width=0.55\textwidth]{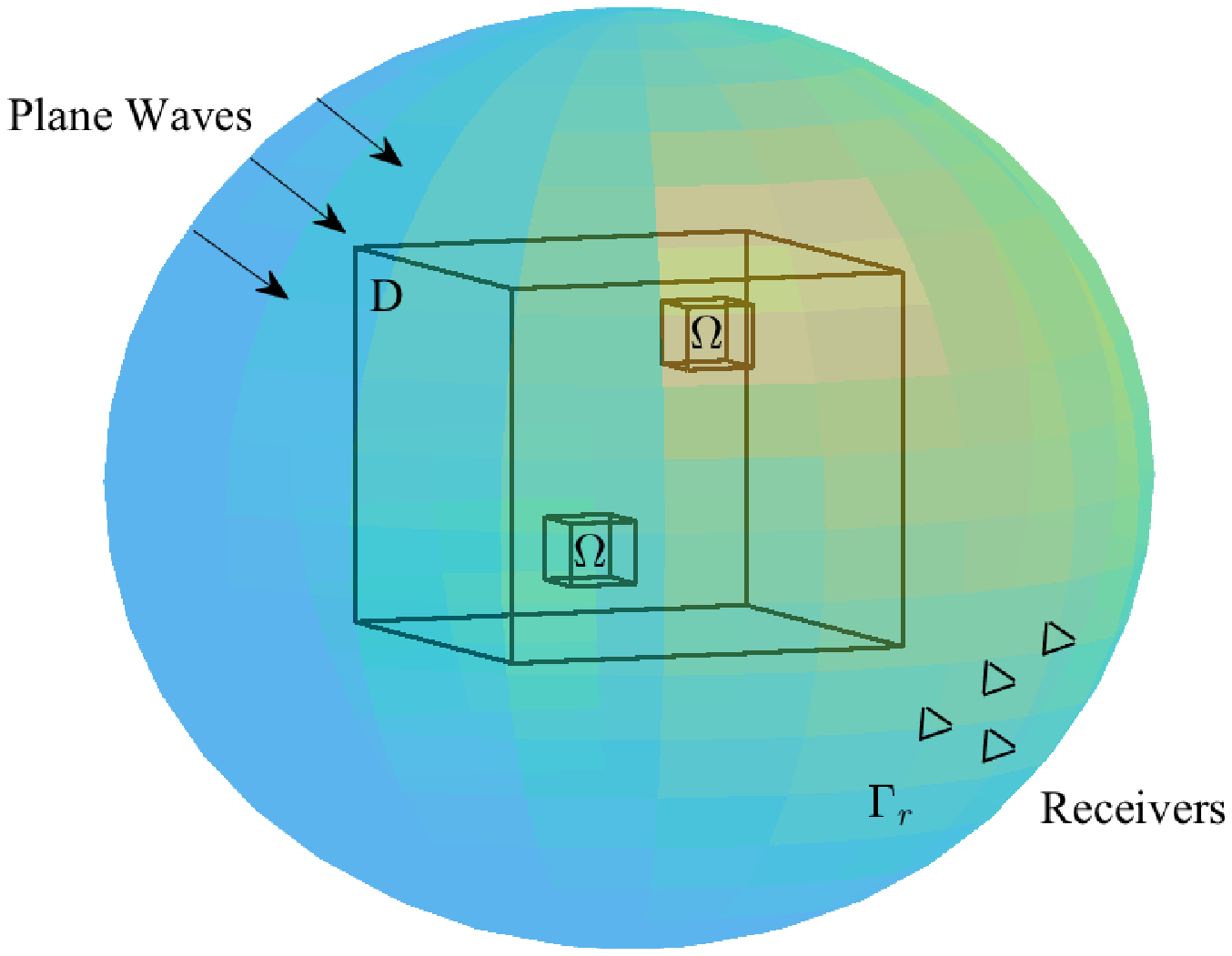}\hfill{}
 \vskip -1truecm
 \caption{\label{fig:demo0}  \small{\emph{Plane wave illumination.}}}
 \end{figurehere}

\subsubsection{Point source illumination} \noindent On the other hand, suppose that $u^i$ is generated from a point source $ y \in \Gamma_s$, where $\Gamma_s$ is a surface. Then, 
\beqnx
u^i(x,y) := G(x,y)= C_{M} k^{\frac{2-M}{2}} |x-z|^{\frac{2-M}{2}} H^{(1)}_{\frac{M-2}{2}}(|x-y|) \,,
\eqnx
which now satisfies 
\beqnx
\Delta u^i+k^2 u^i = \delta_y  \quad \text{for}\quad x\in\mathbb{R}^M \,.
\eqnx
Consider the total field $u$ as the solution to the following equation:
\beqn
\Delta u+k^2\Big(1+q(x)\Big)u = \delta_y  \quad \text{for}\quad x\in\mathbb{R}^M \,.
\label{scattering2}
\eqn
with $u^s := u - u^i$ satisfying the outgoing Sommerfeld radiation 
condition \eqref{sommerfield}.

From \eqref{scattering2}, due to the presence of inclusions in $\Omega$, the total field $u$ satisfies the following equation:
\beqn
u(x,y)=u^i(x,y)+k^2\int_\Omega q(z)G(x,z)u(z,y) \, d z,\quad x\in\mathbb{R}^M\;\;\text{and}\;\; y\in\Gamma_s\,.
\label{compute2}
\eqn
Because the corresponding scattered field $u^s$ is measured by the receivers on the surface $\Gamma_r$, it has the following representation:
 \beqn
u^s(x,y)=k^2\int_\Omega q(z)G(x,z)u(z,y) \, d z,\quad x\in\mathbb{R}^M\;\;\text{and}\;\; y\in\Gamma_s\,.
\label{us2}
\eqn
We refer to Figure \ref{fig:demo} for a schematic configuration of the inverse medium scattering problem.
We may again consider measurements $u(x,y)$ along a surface $x \in \Gamma_r$ where $y \in \Gamma_s$.

\begin{figurehere}
 \vskip 0.5truecm
 \hfill{}\includegraphics[clip,width=0.45\textwidth]{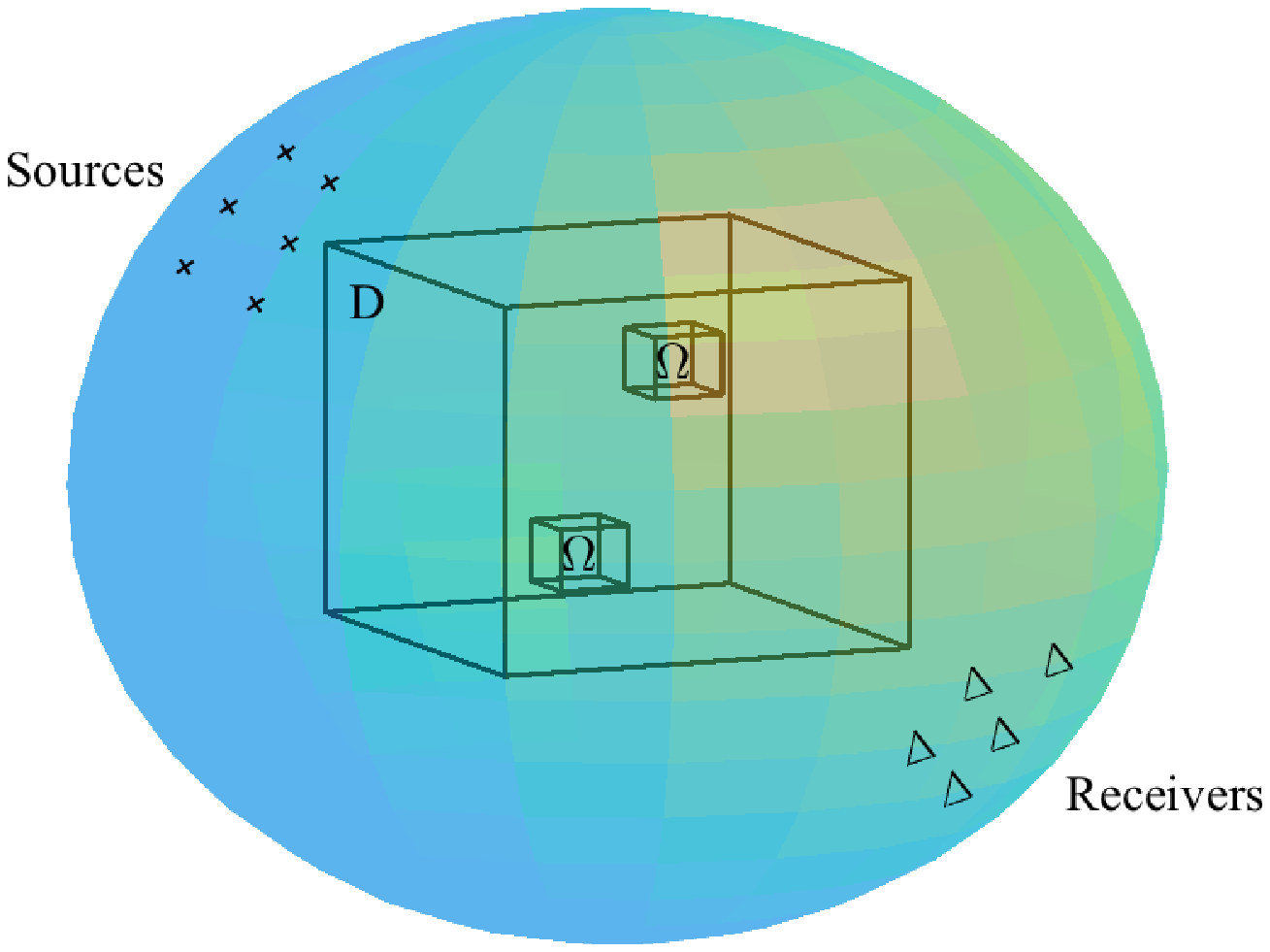}\hfill{}
 \vskip -0.2truecm
 \caption{\label{fig:demo}  \small{\emph{Point source illumination.}}}
 \end{figurehere}

\subsection{Inverse problem} \label{example_haha}

The inverse problem of our interest is to recover the physical features of the inclusions inside $\Omega$ from knowledge of the scattered field $u^s$.  

For {plane wave illuminations}, we consider the following problem:  given the measurements of $u^s(x,d_y)$ along a surface $x \in \Gamma_r$, where the distance $\text{dist} ( \Gamma_r, \Omega)  \gg 1$, and $d_y\in\mathbb{S}^{M}$, where $u$ satisfies \eqref{compute1}, we aim at recovering $q(x), x \in \Omega$.

In the case of {point source illuminations}, we consider the following problem:  given the measurements of $u^s(x,y)$ along a surface $x \in \Gamma_r$ and $y\in \Gamma_s $, where $u$ satisfies \eqref{compute2}, we aim at recovering $q(x), x \in \Omega$.


Note that there are many numerical methods addressing the number of required directions of incidence, see, for instance, \cite{BerBro, music, CK,  LXZ3, LXZ2}.
Since our work does not focus on any particular type of reconstruction methods, but aims at discussing how to choose the mesh size in general, we do not advocate any particular choice of numerical reconstruction methods. Nevertheless, in this subsection, for the sake of completeness, we state two examples of indicator functions in order to illustrate our main ideas.  In what follows, we provide a brief description of two types of indicator functions: index functions provided in the multilevel sampling method (MSM) \cite{L1, L2, LZ}, and the direct sampling method (DSM), see, for instance, \cite{Jin12, L1, LXZ3}.  Both methods work either with plane wave or point source illuminations; nonetheless, for the sake of simplicity, we only consider the case of plane wave illuminations. \\

\noindent \textbf{Example 1:} 
We would like to briefly describe the index function shown in the MSM.  
Given a choice of sampling points $x \in \Omega$.
Consider an incident direction $d_y \in \mathbb{S}^{M-1}$. We first introduce the contrast source function $\phi$, which is defined as follows,
\[\phi(x,d_y):= k^2 q(x) \, u(x,d_y) \quad \text{for}\;\;x\in \Omega.\]
The scattered field $u^s$ can thus be expressed into the following compact form
\[u^s(x_r,d_y)=\mathcal{P}_r\phi \,(x_r,d_y)  \quad \text{for}\;\;x_r\in \Gamma_r,\]
where the integral operator $\mathcal{P}_r$ is given, for $x_r \in \Gamma_r$, by
\[\mathcal{P}_r\varphi\,(x_r,d_y):=\int_\Omega G(x,x_r)\phi(x,d_y) dx.\]
The MSM aims at approximating the contrast source function at sampling points $x \in \Omega$ given by a mesh with multiple choices of $d_y$ via a backpropagation method; and then refine the choices of mesh sizes in order to obtain iteratively another set of $x \in \Omega$.
A backpropagation $\phi_b(\cdot, d_y)$, which is defined as follows, is employed to provide an (initial) approximation of $\phi(x,d_y)$:
\[
\phi_b : =\frac{\big\lVert \mathcal{P}_r^* u^s  \big\rVert_{L^2(\Omega)}^2}{\big\lVert  \mathcal{P}_r\,\mathcal{P}_r^* u^s  \big\rVert_{L^2(\Gamma_r)}^2}\mathcal{P}_r^*u^s.
\]
Here, $\mathcal{P}_r^*$ is known as the backpropagation operator which is represented as
\[\mathcal{P}_r^*\,\phi(x; d_{y}):=\int_{\Gamma_r}\overline{G(\xi;x)}\phi(\xi;d_y) \, d \sigma_\xi \quad \text{for}\;\; x\in \Omega.\]
In practice, the expression of contrast source function is numerically obtained by
\begin{equation}\label{eq:varphi}
 \tilde{\phi}_{b,n} (x) :=\frac{\big\lVert \mathcal{P}_r^* u^s(\cdot,  d_{y,n})  \big\rVert_{L^2(\Omega)}^2}{\big\lVert  \mathcal{P}_r\,\mathcal{P}_r^* u^s(\cdot,  d_{y,n})  \big\rVert_{L^2(\Gamma_r)}^2}\mathcal{P}_r^*u^s(x,  d_{y,n})  \quad \text{for}\;\; x\in \Omega,
 \end{equation}
for a finite number of incidences $d_{y,n} \in \mathbb{S}^{M-1}$, where $n=1,2\cdots,N$. With this in hand, we now define an index function or indicator function $\eta(x)$, which aims at numerically approximating $k^2 q(x)$ from the measured values of $\{ u^s(\cdot ,d_{y,n}) \}_{n=1}^N $, as the least-square solution to the system 
\begin{equation}
 \tilde{\phi}_{b,n} = \eta(x) \, u(x,d_{y_n}), 
 \end{equation}
for $n=1,2\cdots,N$. Here, the points $ x\in \Omega$ are the chosen sampling points in $\Omega$, and $\eta(x)$ can be explicitly given by
\begin{equation}\label{eq:eta}
\eta(x) :=\Re\left\{\frac{\sum\limits_{n=1}^N  \tilde{\phi}_{b,n} (x) \overline{\big(u^i(x, d_{y,n})+\mathcal{P}_\Omega  \tilde{\phi}_{b,n} \big)}}{\sum\limits_{n=1}^N\Big|u^i(x, d_{y,n})+\mathcal{P}_\Omega  \tilde{\phi}_{b,n}\Big|^2}\right\}\quad \text{for}\;\; x\in\Omega,
\end{equation}
where $\Re\{\cdot\}$ denotes the real part of a complex number, $ \tilde{\phi}_{b,n}$ are defined in \eqref{eq:varphi}, and the integral operator $\mathcal{P}_\Omega$ is defined, for all $x \in \Omega$, by
\[\mathcal{P}_\Omega\phi\,(x, d_y):=\int_\Omega G(\xi,x)\phi(\xi,d_y) d\xi.\]
Now the algorithm chooses a certain threshold of cutoff to numerically identify the support of the index function or indicator function $\eta$, and then the choices of $x \in \Omega$ are re-sampled, and the whole process is preformed again and iterated until a fine reconstruction profile is achieved. \\

\noindent \textbf{Example 2:} 
Another example of indicator functions is provided by the DSM.  For a given incidence $d_y \in \mathbb{S}^M$, the index function is defined by
\begin{equation}\label{eqn:indexfcn}
I(x, d_y):=\frac{\Big| \langle u^s(\cdot, d_y),G(\cdot,x) \rangle_{L^2(\Gamma_r)}\Big|}{||u^s(\cdot, d_y)||_{L^2(\Gamma_r)}||G(\cdot,x)||_{L^2(\Gamma_r)}}, \,\quad \forall\,x\in \Omega\,,
\end{equation}
where the inner product $ \langle \cdot, \cdot \rangle_{L^2(\Gamma_r)}$ is given by
\[ \langle u^s(x_r, d_y),G(x_r,x) \rangle_{L^2(\Gamma_r)} : =\int_{\Gamma_r}u^s(x_r, d_y) \overline{G(x_r,x)}ds\,\]
and
\[||u^s(\cdot, d_y)||_{L^2(\Gamma_r)}=\bigg(\int_{\Gamma_r}\Big|u^s(x_r,d_y)\Big|^2 \, d \sigma_{x_r} \bigg)^{\frac{1}{2}},\quad ||G(\cdot,x)||_{L^2(\Gamma_r)}=\bigg(\int_{\Gamma_r}\Big|G(x_r,x)\Big|^2 \, d \sigma_{x_r} \bigg)^{\frac{1}{2}}.\]
In order to fully utilize measurements from multiple measurement events, i.e., with incidences $d_{y,n} \in \mathbb{S}^{M}, n = 1,...,N$, we define an index function $I(x)$ to combine different values of $I(x, d_{y,n})$ in the following way:
\[I(x)=\max_{n=1,..., N}\Big\{I(x, d_{y,n})\Big\}, \quad \forall x\in \Omega.\]
Here, $I_{m}$ means the index function for the $m$th data set.
 It has been discussed theorectically and illustrated numerically in \cite{LXZ3} that the function in \eqref{eqn:indexfcn} has 
a relative large value when a sampling point $x$ is close to an inhomogeneous inclusion in $\Omega$ and decays rapidly as $x$ moves away from any of the inclusions. \\

We notice that in the two aforementioned examples, the choices of sampling points $x$ places a crucial row in constructing the indicator functions (either $\eta(x)$ or $I(x)$ in the previous two examples), and appropriate choices of the mesh size (or the density of sampling points) are of  importance for insuring the best possible quality of the reconstruction procedure.

\section{Optimal mesh size with far-field data under plane wave illumination} \label{sec:2}

In this section, we shall consider the optimal choice of the mesh size for the reconstruction of inclusions inside $\Omega$ from the far-field data.
We consider the case when $\Gamma_r = R \,  \mathbb{S}^{M-1}$ where $R \gg 1$. In this case, instead of having the scattered field $u^s$ as measurement, we may assume that we have obtained our measurement as the far-field pattern $u_\infty$ given by
\beqn 
u_\infty (d_x,d_y) = \lim_{|x| \rightarrow \infty} |x |^{\frac{M-1}{2}} \exp^{-i k |x|}  u_s( |x| d_x, d_y ),
\eqn
where $d_x$ and $d_y$ belong to the unit sphere $\mathbb{S}^{M-1}$.

Hence, we assume that the measurements $u_\infty (d_x,d_y)$, where $d_x, d_y \in \mathbb{S}^{M-1}$,  are performed.   For small wave number $k$, the Born approximation shows that the far-field pattern can be expressed as
\beqn\label{eq:far-field}
u_\infty (d_x,d_y) = \tilde{C}_{k,M} k^2 \int_\Omega q(z)e^{ik(d_x-d_y)\cdot z}dz + O (k^{4} |C_{k,M} |^2),
\eqn
where there error term is in $L^2$, and $$\tilde{C}_{k,M} = \frac{-i}{\sqrt{8 \pi}} \left( \frac{k}{2 \pi} \right)^{\frac{M-2}{2}} \exp \left( - \frac{(M-1)}{4} \pi i \right)$$ is a constant depending on the space dimension $M$ and  $k$.
We have the following result from the Nyquist-Shannon sampling theorem. 
\begin{Theorem}
Let $k$ be fixed and small.
Denote by $\mathfrak{F}(q)(\xi)$ the Fourier transform of $q(y)$:
\beqn
\mathfrak{F}(q)(\xi):=\int_{\Omega}q(z)e^{i \xi\cdot z}dz\quad \text{for} \quad \xi \in \mathbb{R}^M\,.
\eqn
Suppose that $\mathfrak{F}(q)(\xi) = 0 $ for all  $|\xi| > 2 k$. Then, the reconstruction of $q$ from $u_\infty (d_x,d_y)$, where $d_x, d_y \in \mathbb{S}^{M-1}$,  is with an error of order $O(k^2 |\tilde{C}_{k,M} |)$.  The reconstruction can be performed with a choice of the mesh size $h <  \f{\lambda}{2}$, where $\lambda :=2\pi/k$ is the operating wavelength.
\end{Theorem}

\begin{proof}
It is direct to observe that
\beqn
\mathfrak{F}(q)(\xi) =  \frac{1}{\tilde{C}_{k,M} k^2}  u_\infty (d_x,d_y) + \mathcal{O} (k^{2} |\tilde{C}_{k,M}|) , 
\eqn
where $\xi:= k ( d_x - d_y)  \in\{\xi:|\xi|\leq 2 k\}$. Hence, from $u_\infty (d_x,d_y), \, d_x, d_y \in \mathbb{S}^{M-1}$, we obtain the values of $\mathfrak{F}(q)(\xi)$ for $|\xi|\leq 2 k$ up to an error of the order $\mathcal{O} (k^{2} |\tilde{C}_{k,M}|)$ in $L^2$. Since $\mathfrak{F}(q)(\xi) = 0 $ for all  $|\xi| > 2 k$, we have
$\mathfrak{F}(q)(\xi)$  up to an error of the order $\mathcal{O} (k^{2} |\tilde{C}_{k,M}|)$ in $L^2$.
Applying the inverse Fourier transform, we obtain $q$ up to an error of the order $\mathcal{O} (k^{2} |\tilde{C}_{k,M}|)$ in $L^2$.  An application of the Nyquist-Shannon sampling theorem yields the last statement in the theorem with a choice of $h < \f{\lambda}{2}$.
\end{proof}
Notice that the above result shows that we may maximize our choice of the mesh size by taking $h = \f{\lambda}{2} - \epsilon$ for a small $\epsilon$ in order to minimize the computation complexity.

In order to recover the contrast function $q(x)$, we should discretize \eqref{eq:far-field} into the following form:
\[u_\infty\approx \gamma k^2\sum_{j=1}^N s_j \,q(z_j)e^{ik(d_x-d_y)\cdot z_j}.\]
Here,  the optimal $N$ is $\big\lceil 4|\Omega|/\lambda^2 \big\rceil$ when $M = 2$, where $\lceil x\rceil$ means the largest integer that is less than or equal to $x$, $s_j$ denotes the $j$-th grid point in $\Omega$, and $|\Omega|$ represents the area (volume) of $\Omega$. 


\section{Optimal mesh size with near-field data under point source illumination}  \label{sec:3}
In this section, optimal mesh size for contrast reconstruction from near-field data will be discussed.
Based on the first term of the Neumann series solution to the Lippmann-Schwinger integral equation which is known as the Born approximation, the scattered field $u^s$ is approximated by
\beqn
u^s_B(x,y)=k^2\int_{\Omega} q(z)G(x,z)u^i(y,z) \, d z + O(k^4),\quad x\in\Gamma_r\;\;\text{and}\;\; y\in\Gamma_s\,.
\eqn
Furthermore, the incident field $u^i$ in this case is provided by the Green function. Thus, the estimated scattered field $u^s_B$ can be written into the following multi-static response (MSR):
\beqn
[\text{MSR}(q)](x,y)  := \frac{1}{k^2} u^s_B(x,y)  =\int_{\Omega} q(z)G(x,z)G(z,y) \, d z + O(k^2) ,\quad x\in\Gamma_r\;\;\text{and}\;\; y\in\Gamma_s\,.
\eqn

In practice, we again discretize the sampling domain $D$ into $N$ sub-regions and denote the center of each sub-region  by $z_j$ for $j=1,2,\cdots, N$. Based on this discretization, we have the following MSR approximation operator $A$ up to a normalization constant $\hat{C}_{k,M}$ depending on $k$ and $M$, which is the dominant part of the near-field data,
\beqn{\label{A}}
A :=  
\sum \limits_{j=1}^N q(z_j)  |z_j-x|^{\frac{2-M}{2}}  |z_j-y|^{\frac{2-M}{2}}  H^{(1)}_{\frac{M-2}{2}}(k|z_j-x|) H^{(1)}_{\frac{M-2}{2}}(k|z_j -y|)   \quad \text{for} \quad x\in\Gamma_r\;\;\text{and}\;\; y\in\Gamma_s\,,
\eqn
where $z_j \in \Omega$.

Let us consider for simplicity the case when $\Gamma_r = \Gamma_s.$
In order to investigate the optimal mesh size for the reconstruction process, we need to determine the number of observable singular values for the MSR operator $A$ when $z_j$ are close to each other.

For this purpose, we first define a conjugation map on $L^2(\Gamma)$ as follows,
\beqnx
C: L^2(\Gamma) \rightarrow L^2(\Gamma) \quad f \mapsto \bar{f} \,.
\eqnx
Now, from the fact that the operator $A$ is compact and $C$-complex symmetric, namely, $A^* = C A C$. We have the following Danciger's variational principle for the singular values $\{ \sigma_i \}$ of $A$ \cite{Garcia},
\beqn
\sigma_n = \min \limits_{\text{codim}_{\mathbb{R}} V = n} \max \limits_{\substack{p \in V \\ || p || = 1}} {\Re}[A p, p] \,,
\label{varvar}
\eqn
where $[Ap, p]=p^TAp$ and $\sigma_0 \geq \sigma_1 \geq ... \geq 0$.

Before we state the following main result, we denote the locations of receivers and sources by $\Gamma:=\Gamma_r\cup\Gamma_s$.
Firstly, we consider the simplest case when $N=1$. Because $q$ is a constant, we investigate the following operator $A_1$ which is the key part of the operator $A$:
\beqn
A_1 := a  |z-x|^{\frac{2-M}{2}}   |z-y|^{\frac{2-M}{2}}  H^{(0)}_{\frac{M-2}{2}}(k|z -x|) H^{(1)}_{\frac{M-2}{2}}(k|z -y|) = w \otimes w^T\quad \text{for}\quad x,y\in\Gamma,
\eqn
where $a= || |z-x|^{\frac{2-M}{2}}  H^{(1)}_{\frac{M-2}{2}}(k|z - x |)  ||^2_{L^2(\Gamma)}$, $\otimes$ means the Kronecker product and $$w :=  |z-x|^{\frac{2-M}{2}}  H^{(1)}_{\frac{M-2}{2}}(k|z -x|) \bigg/ \Big\| |z-x|^{\frac{2-M}{2}}   H^{(1)}_{\frac{M-2}{2}}(k|z - x |)  \Big\|_{L^2(\Gamma)}$$ is a unit vector.
It is easy to see that
\beqn
{\Re}[A_1 p, p] = {\Re} (p^T A_1 p ) = {\Re} ( w^T p )^2
\begin{cases}
= 0 & \text{ if }\; p \perp_{\mathbb{R}} \overline{w} \,, \\[2mm]
\neq 0 & \text{ otherwise},
\end{cases}
\eqn
and that ${\Re}[A_1 p, p] \leq 1$. We infer from Danciger's principle \cite{Dan} that the singular values of $A_1$ are as follows: $\sigma_0 = 1$ and $\sigma_i=0\, (i\geq1)$.

Secondly, we investigate the case when $N=2$.
With the help of  \eqref{A}, the main part of the MSR matrix is of the form
\beqn
A_2 := b_1w_1 \otimes w_1^T + b_2w_2 \otimes w_2^T , \label{A2}
\eqn
where $b_j= ||   |z_j-x|^{\frac{2-M}{2}}  H^{(1)}_{\frac{M-2}{2}}(k|z_j - x |)  ||^2_{L^2(\Gamma)}$, and $$w_j:= |z_j-x|^{\frac{2-M}{2}}    H^{(1)}_{\frac{M-2}{2}}(k|z_j -x|) \bigg/ \Big\|  |z_j-x|^{\frac{2-M}{2}}    H^{(1)}_{\frac{M-2}{2}}(k|z_j - x |)  \Big\|_{L^2(\Gamma)},$$ for $j = 1,2$, are unit vectors.
Similarly, we are able to directly obtain that 
\beqn
{\Re}[A_2 p, p] = {\Re} (p^T A_2 p ) \begin{cases}
= 0 & \text{ if } p \perp_{\mathbb{R}} \text{span}_{\mathbb{R}}\{ \overline{w_1}, \overline{w_2} \}\,, \\[2mm]
\neq 0 & \text{  otherwise}\,.
\end{cases}
\label{var2}
\eqn
Now, let us define $\{ \overline{h_1} , \overline{h_2} \} := \Big\{ \frac{\overline{w_1  + w_2}}{\|w_1  + w_2\|_{L^2(\Gamma)}} , \frac{\overline{w_1 - w_2}}{\|w_1 - w_2\|_{L^2(\Gamma)}} \Big\}$. Then, we have
\beqn
\| h_1 \|_{L^2(\Gamma)} = \| h_2 \|_{L^2(\Gamma)} = 1 \,, \quad {\Re} \langle h_1, h_2 \rangle_\Gamma  = 0 \,,
\label{var22}
\eqn
and that $\text{span}_{\mathbb{R}}\{ \overline{h_1}, \overline{h_2} \} = \text{span}_{\mathbb{R}}\{ \overline{w_1}, \overline{w_2} \}$.  Hence, the operator $A_2$ can be written as,
\beqn
A_2 =\frac{1}{2}\|w_1  + w_2\|_{L^2(\Gamma)}^2 \,  h_1 \otimes h_1^T + \frac{1}{2} \|w_1  - w_2\|_{L^2(\Gamma)}^2 \, h_2 \otimes h_2^T.
\label{var222}
\eqn
Considering \eqref{var2}-\eqref{var222}, Danciger's principle yields that the singular values of $A_2$ satisfy
\beqn
\sigma_0 &=& \frac{1}{2}\max \Big\{ \|w_1  + w_2\|_{L^2(\Gamma)}^2 , \|w_1  - w_2\|_{L^2(\Gamma)}^2 \Big\} = 1 + |\Re\langle w_1,w_2\rangle|\, , \\[2mm]
\sigma_1 &=& \frac{1}{2}\min \Big\{ \|w_1  + w_2\|_{L^2(\Gamma)}^2 , \|w_1  - w_2\|_{L^2(\Gamma)}^2 \Big\} = 1 - |\Re\langle w_1,w_2\rangle|\, , \\[2mm]
\sigma_l &=& 0,  \quad \quad  \forall\, l >1,
\eqn
where $\langle w_1,w_2\rangle=\int_\Gamma \overline{w_1}w_2d\sigma$. The above results can be summarized in the following theorem.
\begin{Theorem}
Suppose that the operator $A$ can be represented as follows,
\beqn
A := w_1 \otimes w_1^T + w_2 \otimes w_2^T \,,
\eqn
where $w_j$ are unit vectors in $L^2(\Gamma)$ for $j= 1,2$. We have the following expression for the singular values of $A$:
\beqn
\sigma_0 = 1 + |{\Re}\langle w_1,w_2\rangle|\, ,
\sigma_1 = 1 - |{\Re}\langle w_1,w_2\rangle|\, ,
\sigma_2 = \sigma_3 =\cdots= 0,
\eqn
and the quotient $\alpha$ between the two non-zero singular values is expressed as,
\beqn
\alpha := \frac{\sigma_1}{\sigma_0} =  \frac{1 - |{\Re}\langle w_1,w_2\rangle|}{1 + |{\Re}\langle w_1,w_2\rangle|}\,.
\eqn
\end{Theorem}

We have noticed that the ratio $\alpha$ represents the relative observability of the second singular value with respect to the first one. 
Therefore, it quantifies the ability to distinguish between two inhomogeneities from the measurements in the form of the MSR matrix.  Consequently, we introduce the following definition of the $\alpha$-distinguishable with respect to $\Gamma$, for $0 < \alpha < 1$.
\begin{Definition}
Two point scatterers at locations $z_1$ and $z_2$ with equal weights 
are $\alpha$-distinguishable with respect to $\Gamma_r$ if the following inequality is satisfied:
\beqn
1 - \alpha \geq  \frac{1 - |{\Re}\langle {w}_1,{w}_2\rangle|}{1 + |{\Re}\langle {w}_1, {w}_2\rangle|}, 
\label{alpha1}
\eqn
where  $$w_j:= \left( k |z_j-x| \right)^{\frac{2-M}{2}}  H^{(1)}_{\frac{M-2}{2}}(k|z_j -x|) \bigg/ \Big\|  \left( k |z_j-x| \right)^{\frac{2-M}{2}}  H^{(1)}_{\frac{M-2}{2}}(k|z_j - x |)  \Big\|_{L^2(\Gamma)}$$  with $j=1,2$ and $\forall{z}_j\in D$, $x\in\Gamma_r$.
\end{Definition}

Without loss of generality, we consider two point scatterers located at $z_1 = z$ and $z_2 = z + \varepsilon v$, where $v$ is a unit vector and $\varepsilon$ is a distance parameter which can be viewed as the mesh size. We acquire the interval of $\varepsilon$ when the two point scatterers are $\alpha$-distinguishable with respect to $\Gamma$. 
From the series expansion of $(1 + a)^{- \beta}$ for all $\beta \neq 0 $ w.r.t. $a$, we derive that if $ \varepsilon < |z-x| $, then
\beqnx
|z + \varepsilon v - x |^{- \beta} = \sum_{r=0}^\infty \sum_{s=0}^r \frac{\Gamma(\frac{\beta}{2} + r)}{\Gamma(\frac{\beta}{2}) s! (r-s) ! }  (-1)^r  \varepsilon^{r + s} \cos^{r-s} (\theta_{z-x,v} )   |z - x |^{ - \beta - r - s},
\eqnx
where $\theta_{z-x,v}$ denotes the relative angle between $z-x$ and $v$ and that the series converges absolutely and uniformly when $ \varepsilon < (1-\gamma) \min \{ 1, |z-x| \} $,  for some $0 < \gamma < 1$.
Together with Graf's fromula \cite{Watson}, we have if $ \varepsilon <  (1-\gamma) \, \min\{ 1, \text{dist}(\Gamma, z) \} \, $, then the following expansion holds:
\beqn
& & {\Re} \left(\int_{\Gamma}  |z-x|^{\frac{2-M}{2}}    |z + \varepsilon v -x|^{\frac{2-M}{2}}  H^{(2)}_{\frac{M-2}{2}}(k|z -x|) H^{(1)}_{\frac{M-2}{2}}(k|z + \varepsilon v -x|) \, d \sigma_x \right) \notag  \\[2mm]
&= &
\begin{cases}
\sum \limits_{n\in \mathbb{Z}} p_{n} J_n(k \varepsilon )  & \text{ if } M = 2, \\[3mm]
\sum \limits_{n\in \mathbb{Z}}  \sum \limits_{r \in \mathbb{N}} \sum \limits_{s=0}^r   \, p_{nrs} \,  J_n(k \varepsilon ) \, \varepsilon^{r+s} & \text{ if } M > 2,
\end{cases}
\label{seriesA}
\eqn
where $p_n$ and $p_{n,r,s}$ are explicitly given by
{\small
\beqn
p_n &:=& {\Re} \left(\int_{\Gamma}  H^{(2)}_{0}(k|z -x|) H^{(1)}_{n}(k|z -x|) e^{- i n \theta_{z-x,v}} \, d \sigma_x \right)\,
\label{pn} \\[2mm]
p_{nrs} &:= & (-1)^r   \frac{\Gamma(\frac{M-2}{4} + r)}{\Gamma(\frac{M-2}{4}) s! (r-s) ! } \,  {\Re} \bigg(\int_{\Gamma}  |z-x|^{2-M - r - s}   H^{(2)}_{\frac{M-2}{2}}(k|z -x|) \notag \\[2mm]
      &&\cdot H^{(1)}_{\frac{M-2}{2} + n}(k|z -x|) \cos^{r-s} (\theta_{z-x,v} )  e^{- i n \theta_{z-x,v}} \, d \sigma_x \bigg)\,\label{pnrs}
\eqn
}with $\theta_{z-x,v}$ denoting the relative angle between $z-x$ and $v$. Likewise, we have
\beqn
& &\int_{\Gamma}  |z + \varepsilon v -x|^{2-M}  H^{(2)}_{\frac{M-2}{2}}(k|z + \varepsilon v -x|) H^{(1)}_{\frac{M-2}{2}}(k|z + \varepsilon v -x|) \, d \sigma_x   \notag\\[2mm]
&= &
\begin{cases}
\sum \limits_{n,m\in \mathbb{Z}} h_{nm} J_n(k \varepsilon )  J_m(k \varepsilon ) & \text{ if } M = 2 ,\\[3mm]
\sum \limits_{n,m\in \mathbb{Z}} \sum \limits_{r \in \mathbb{N}} \sum \limits_{s=0}^r  h_{nmrs}  \, J_n(k \varepsilon )  J_m(k \varepsilon ) \, \varepsilon^{r+s} & \text{ if } M > 2,
\end{cases}
\label{seriesB}
\eqn
where $h_{nm}$ and $h_{nmrs}$ are given by
{\small
\beqn
h_{nm} &:=& {\Re} \left(\int_{\Gamma} H^{(2)}_{m}(k|z -x|) H^{(1)}_{n}(k|z -x|) e^{- i (n-m)\theta_{z-x,v}} \, d \sigma_x \right)\,,
\label{hnm} \\[2mm]
h_{nmrs}  &:=&(-1)^r
\frac{\Gamma(\frac{M-2}{2} + r)}{\Gamma(\frac{M-2}{2}) s! (r-s) ! }     
 {\Re} \bigg(\int_{\Gamma} |z - x |^{ 2 - M - r - s} H^{(2)}_{\frac{M-2}{2}+m}(k|z -x|)\notag\\[2mm]
 &&\cdot H^{(1)}_{\frac{M-2}{2} + n}(k|z -x|)  \cos^{r-s} (\theta_{z-x,v} )  e^{- i (n-m)\theta_{z-x,v}} \, d \sigma_x \bigg)\,.\label{hnmrs}
\eqn
}Assume that $0 <R_0 \leq \text{dist}(z,\Gamma) \leq R_1$ for some $R_0, R_1 \in \mathbb{R}$, and consider the asymptotic property of $J_m$ \cite{handbook}. We can infer that both the series \eqref{seriesA} and \eqref{seriesB} converge absolutely.
If we further impose the restriction that $\frac{k \, \varepsilon}{2} \leq 1$, then from the following well-known Taylor series of the Bessel functions \cite{Watson},
\beqn
J_n(k \varepsilon ) = (\text{sgn}(n))^{n} \sum_{l \in \mathbb{N}} \frac{(-1)^{l}}{l! (l+|n|)!} \left(\frac{k \,  \varepsilon}{2}\right)^{2 l + |n|} \,,
\eqn
we immediately obtain the following expansion of the inner product:
{\small
\begin{eqnarray}
& & {\Re} \left(\int_{\Gamma}  |z-x|^{\frac{2-M}{2}}    |z + \varepsilon v -x|^{\frac{2-M}{2}}  H^{(2)}_{\frac{M-2}{2}}(k|z -x|) H^{(1)}_{\frac{M-2}{2}}(k|z + \varepsilon v -x|) \, d \sigma_x \right)
\notag \\[2mm]
&=& 
\begin{cases}
\sum \limits_{l \in \mathbb{N}} \sum \limits_{n\in \mathbb{Z}} (\text{sgn}(n))^{n} p_n \frac{(-1)^{l}}{l! (l+|n|)!} \left(\frac{k \,  \varepsilon}{2}\right)^{2 l + |n|}  \;\text{ for } M = 2 ,\\[3mm]
\,
\sum \limits_{l \in \mathbb{N}}  \sum \limits_{n\in \mathbb{Z}}  \sum \limits_{r \in \mathbb{N}} \sum \limits_{s=0}^r   \, (- \text{sgn}(n))^{n} p_{nrs}  \frac{(-1)^{l}}{l! (l+|n|)!} \left(\frac{k \,  \varepsilon}{2}\right)^{2 l + |n|} \varepsilon^{r+s} \;  \text{ for } M > 2,
\end{cases}
\label{series1}
\\[2mm]
& &\int_{\Gamma}  |z + \varepsilon v -x|^{2-M}  H^{(2)}_{\frac{M-2}{2}}(k|z + \varepsilon v -x|) H^{(1)}_{\frac{M-2}{2}}(k|z + \varepsilon v -x|) \, d \sigma_x  \notag \\[2mm]
&=&
\begin{cases}
 \sum \limits_{l, p \in \mathbb{N}} \sum \limits_{n,m \in \mathbb{Z}}  (\text{sgn}(n))^{n}  (\text{sgn}(m))^{m} h_{nm} \frac{(-1)^{l+p}}{l!p! (l+|n|)! (p+|m|)!} \left(\frac{k \,  \varepsilon}{2}\right)^{2 (l+p) + |n| + |m| } \; \text{ for } M = 2 , \\[3mm]
 \sum \limits_{l, p \in \mathbb{N}} \sum \limits_{n,m \in \mathbb{Z}} \sum \limits_{r \in \mathbb{N}} \sum \limits_{s=0}^r   (\text{sgn}(n))^{n}  (\text{sgn}(m))^{m} h_{nmrs} \frac{(-1)^{l+p}}{l!p! (l+|n|)! (p+|m|)!} \left(\frac{k \,  \varepsilon}{2}\right)^{2 (l+p) + |n| + |m| } \varepsilon^{r+s} \; \text{ for } M > 2,
\end{cases}
\label{series2}
\end{eqnarray}
}where the two series at the right-hand side converge absolutely for $\frac{k \, \varepsilon}{2} \leq 1$ and  $ \varepsilon <   (1-\gamma)\, \min \{ 1, \text{dist}(\Gamma, z) \}$.

\subsection{$\alpha$-distinguishability when $M=2$}

For simplicity, let us first focus on the case when $M = 2$. We obtain the following estimate,
{\small
\beqn
\left| {\Re} \left(\int_{\Gamma} H^{(2)}_{0}(k|z -x|) H^{(1)}_{0}(k|z + \varepsilon v -x|) \, d \sigma_x \right)- p_0 + \frac{k \,  \varepsilon}{2} ( p_1 - p_{-1})  \right| &\leq C_1 \frac{k ^2\,  \varepsilon^2}{4}, \quad \quad  \label{approx1}\\[2mm]
\left| \int_{\Gamma} H^{(2)}_{0}(k|z + \varepsilon v -x|) H^{(1)}_{0}(k|z + \varepsilon v -x|) \, d \sigma_x -
h_{00} - \frac{k \,  \varepsilon}{2} \sum_{|m|+|n|=1} (\text{sgn}(n))^{n}  (\text{sgn}(m))^{m}  h_{nm}  \right| &\leq C_2 \frac{k^2 \,  \varepsilon^2}{4},\quad \quad  \label{approx2}
\eqn}
where $p_0 = h_{00} =  || H^{(1)}_{0}(k|z - x |)  ||^2_{L^2(\Gamma)}$ and $C_1$ and $C_2$ are two finite constants having the following forms:
\beqn
C_1 &=& \sum_{l>0} \sum \limits_{n\in \mathbb{Z}}  \frac{|p_n|}{l! (l+n)!} + \sum \limits_{n>1}  \frac{|p_n|}{ n !}
\label{constant1} \,,\\[2mm]
C_2 &=& \sum_{l+p >0} \sum \limits_{m,n\in \mathbb{Z}}  \frac{|h_{nm}|}{l!p! (l+n)! (p+m)!} +\sum \limits_{m+n>1} \frac{|h_{nm}|}{ n! m!} \,
\label{constant2}
\eqn
if  $\frac{k \, \varepsilon}{2} \leq 1$ and  $ \varepsilon < (1-\gamma)\, \min \{ 1, \text{dist}(\Gamma, z) \} $.
From \eqref{approx1} and \eqref{approx2}, we have the following approximation:
\beqn
1 - \left| {\Re}\langle w_1,w_2\rangle \right|^2 = 1 - \frac{\left(p_0 - ( p_1 - p_{-1}) \frac{k \,  \varepsilon}{2} + \mathcal{O}\left(\frac{k^2 \,  \varepsilon^2}{4} \right) \right)^2}{ p_0 \left(p_0 + \frac{k \,  \varepsilon}{2}\sum\limits_{|m|+|n|=1}  (\text{sgn}(n))^{n}  (\text{sgn}(m))^{m}   h_{nm}  + \mathcal{O}\left(\frac{k^2 \,  \varepsilon^2}{4}\right)  \right) } \,.
\label{taylortaylor}
\eqn
We further assume that
\beqn
    \frac{k \,  \varepsilon}{2} \leq  \sqrt{ \frac{p_0^2 k_1^2 }{4 C_2^2} + \frac{p_0 (1 - \delta)}{C_2} } -  \frac{p_0 |k_1| }{2 C_2}
    \label{bounddelta}
\eqn
for some $ 0 < \delta < 1$, where 
\beqn
k_1= \frac{\sum_{|m|+|n|=1} (\text{sgn}(n))^{n}  (\text{sgn}(m))^{m}  h_{nm} }{p_0} \,. 
\label{k1}
\eqn
Then, we have
\beqn
\max \left\{ \frac{k \,  \varepsilon}{2}   \left| k_1 - \frac{C_2}{p_0} \frac{k \,  \varepsilon}{2}  \right|,  \frac{k \,  \varepsilon}{2}   \left| k_1  +  \frac{C_2}{p_0} \frac{k \,  \varepsilon}{2}  \right| \right\} \leq 1 - \delta,
\eqn
and thus, we can derive from the Taylor expansion and \eqref{taylortaylor} that
\beqn \label{estima}
\left| 1 - |{\Re}\langle w_1,w_2\rangle|^2 - \left( \frac{ 2 (p_1 - p_{-1})}{p_0} + k_1 \right)\frac{k \,  \varepsilon}{2}
\right| \leq C \frac{k^2 \,  \varepsilon^2}{4}
\,
\eqn
with the constant $C$ given by
{\small
\beqn
C = \frac{C_2}{p_0}  \left(1 + 2 \left|k_1 \right| \delta^{-1} \right) +  k_1^2 \delta^{-1} +  \frac{C_2^2}{p_0^2} \delta^{-1}
+ \frac{ 2 ( | p_1 | + |p_{-1}| ) | k_1| }{p_0}
+ \frac{ ( | p_1 | + |p_{-1}|  )^2 + 2 (p_0 +  | p_1 | + |p_{-1}|  )C_1 + C_1^2}{p_0^2}, \quad 
\label{constant3}
\eqn
}when $0<\delta<1$.
We can rewrite the condition in \eqref{alpha1} as follows,
\beqn
\frac{4  (1- \alpha) }{(2 - \alpha)^2}  \geq 1 - |{\Re}\langle {w}_1,{w}_2\rangle|^2.
\eqn
With the help of \eqref{estima}, the two point inclusions with the same magnitude 
are $\alpha$-distinguishable if the following inequality holds:
\beqn
\frac{4 (1- \alpha)}{(2 + \alpha)^2}  \geq  \left( \frac{ 2 ( p_1 - p_{-1})}{p_0} + k_1 \right)\frac{k \,  \varepsilon}{2} + C \left( \frac{k \,  \varepsilon}{2}\right)^{2} \,.
\eqn
This directly infers the following theorem.
Note that we may always assume that $k \varepsilon < 1/2$.

\begin{Theorem} \label{important_theorem}
When the dimension $M = 2$, consider a point $z$ such that $0 <R_0 \leq \text{dist}(z,\Gamma) \leq R_1$ for some $R_0, R_1 \in \mathbb{R}$, where $\Gamma = \Gamma_s = \Gamma_r$, $ \frac{k \,  \varepsilon}{2} \leq  \min \left\{ \frac{1}{4} , \sqrt{ \frac{p_0^2 k_1^2 }{4 C_2^2} + \frac{p_0 (1 - \delta)}{C_2} } -  \frac{p_0 |k_1| }{2 C_2} \right\}$ for some $0< \delta <1 $,  $ \varepsilon < (1-\gamma) \, \min \left\{ 1,  \text{dist}(\Gamma, z) \right \} $ for some $ 0< \gamma <1 $, $p_n$ are defined as in \eqref{pn}, $h_{nm}$ as in \eqref{hnm}, and let $C_1,C_2,C, k_1$ be stated in \eqref{constant1}, \eqref{constant2} \eqref{constant3} and \eqref{k1}, respectively.
The two point scatterers with equal weights located at $z$ and $z+ \varepsilon v$ are $\alpha$-distinguishable with respect to $\Gamma$ if the following inequality holds:
\beqn
 \frac{k \,  \varepsilon}{2}
\leq
- \left( \frac{ 2 p_1 - 2 p_{-1} + k_1 p_0 }{ 2 C p_0}  \right) + \sqrt{ \left( \frac{ 2 p_1 - 2 p_{-1} + k_1 p_0}{ 2 C p_0} \right)^2 +  \frac{16 (1- \alpha)}{ C (2 + \alpha)^2}  } \,.
\label{conclusion}
\eqn
\end{Theorem}

Unfortunately, both the conditions and conclusions of the above theorem are too complicated to be practical for choosing an optimal mesh size. In order to obtain an optimal mesh size for practical use, we provide  approximations for the coefficients in the above inequalities: we estimate $C_1$ and $C_2$  by
\beqn
C_1 \approx p_0
\, , \quad
C_2 \approx 2 p_0 \, .
\label{constantapprox}
\eqn
By considering \eqref{series1} and  \eqref{series2}, and further imposing that $\delta = 1/2$, the following approximation of the constant $C$ holds:
{\small
\beqn
C \approx  13  + 8 \left|k_1 \right|  +  2 k_1^2 
+ \frac{ ( | p_1 | + |p_{-1}|  )^2 + 2 ( | p_1 | + |p_{-1}|  ) (1 + |k_1| )p_0 }{p_0^2}.
\label{constantapprox2}
\eqn}
The condition \eqref{bounddelta} with $k \varepsilon /2 < 1/4$ can also be simplified to be the above approximate condition
\beqn
    \frac{k \,  \varepsilon}{2} <   \min \left\{ \frac{1}{4}, \sqrt{ \frac{k_1^2 }{16} + \frac{1}{4} } -  \frac{|k_1| }{4} \right\}.
\eqn
Combining the above estimations of the constants $C_1,C_2,$ and $C$ with \eqref{conclusion}, we obtain the following approximate description of an \textbf{anisotropic} optimal density  at any point $z$ and any direction $v$ for contrast reconstruction from the near-field data in terms of the $\alpha$-distinguishably,
\beqn
h_{z,v}(\alpha) :=  \frac{2}{k}\min \left\{ \frac{1}{4} , \sqrt{ \frac{k_1^2 }{16} + \frac{1}{4} } -  \frac{|k_1| }{4}  ,
    \left( \frac{ 2 p_1 - 2 p_{-1} + k_1 p_0 }{ 2 C p_0}  \right) - \sqrt{ \left( \frac{ 2 p_1 - 2 p_{-1} + k_1 p_0}{ 2 C p_0} \right)^2 + \frac{16 (1- \alpha)}{ C (2 + \alpha)^2}  } 
    \right\}.
\label{finalfunction} 
\eqn
Note that  $p_0, p_1, p_{-1}, k_1,$ and $C$ depend on $z$, $v,$ and $\Gamma$.

\begin{Remark}
A few remarks are in order: 
\begin{enumerate}
\item[(i)]
In practice, the measurements are taken over a finite number of points on $\Gamma = \Gamma_r = \Gamma_s$ and hence, the integrals in \eqref{pn},  \eqref{pnrs}, \eqref{hnm}, and \eqref{hnmrs} can be evaluated numerically as a sum over the points.  
\item[(ii)]
In practice, small regions with fine details are usually anticipated to be recovered with high resolution. Thus, a uniform discretization in the whole sampling domain may not be the best.  
Hence, we may apply adaptivity, and use the optimal mesh size in \eqref{finalfunction} in the region with the requirements of high resolution and a coarse mesh to the other regions. 
\item[(iii)]
To get the interpretation of the optimal mesh size correct, actually it means if you refine further than the mesh size given, it will not improve the resolution.  Hence, if the object is far away, then it is almost identical to using far-field reconstruction, and we do not need to refine better than $\lambda/2$.
\item[(iv)]
A similar analysis can be performed in the general case when $\Gamma_r$ and $\Gamma_s$ may not coincide, but we skip the analysis for the sake of simplicity.
\end{enumerate}
\end{Remark}

\subsection{$\alpha$-distinguishability when $M > 2$}

Similarly, when $M > 2$, we have,
{\footnotesize
\beqn
\left| {\Re} \left(\int_{\Gamma} H^{(2)}_{0}(k|z -x|) H^{(1)}_{0}(k|z + \varepsilon v -x|) \, d \sigma_x \right)- p_{000} + \frac{k \,  \varepsilon}{2}   \left(p_{1\,0\,0} - p_{-1\,0\,0}  - \frac{2 p_{0\,1\,0}}{k}  \right)  \,   \,  \right| &\leq C_{1,M} \frac{k ^2\,  \varepsilon^2}{4}, \quad \quad   \label{approx1M}\\[3mm]
\left| \int_{\Gamma} H^{(2)}_{0}(k|z + \varepsilon v -x|) H^{(1)}_{0}(k|z + \varepsilon v -x|) \, d \sigma_x -
h_{0000} + \frac{k \,  \varepsilon}{2}  \left(\sum_{|m|+|n|=1} (\text{sgn}(n))^{n}  (\text{sgn}(m))^{m}  h_{nm00} +  \frac{2 h_{0010} }{k} \right)  \right| &\leq C_{2,M} \frac{k^2 \,  \varepsilon^2}{4}, \quad \quad  \label{approx2M}
\eqn
}where $p_{000} = h_{0000} =  \Big\| | z-x |^{\frac{M-2}{2}} H^{(1)}_{\frac{M-2}{2}}(k|z - x |)  \Big\|^2_{L^2(\Gamma)}$ and $C_{1,M}$ and  $C_{2,M}$ are two constants given if  $\frac{k \, \varepsilon}{2} \leq 1$ and  $ \varepsilon < (1-\gamma)\, \min \{ 1, \text{dist}(\Gamma, z) \} $ by
{\footnotesize
\beqn
C_{1,M} &=& \sum \limits_{l>0}   \sum \limits_{n\in \mathbb{Z}}  \sum \limits_{r \in \mathbb{N}} \sum \limits_{s=0}^r    \frac{|p_{nrs}|}{l! (l+n)!} + \sum \limits_{n>1}   \sum \limits_{r \in \mathbb{N}} \sum \limits_{s=0}^r    \frac{|p_{nrs}|}{ n !}
\label{constant1M} \,,\\[3mm]
C_{2,M} &=& \sum_{l+p >0} \sum \limits_{m,n\in \mathbb{Z}}  \sum \limits_{r \in \mathbb{N}} \sum \limits_{s=0}^r     \frac{|h_{nmrs}|}{l!p! (l+n)! (p+m)!} +\sum \limits_{m+n>1}  \sum \limits_{r \in \mathbb{N}} \sum \limits_{s=0}^r    \frac{|h_{nmrs}|}{ n! m!} \,.
\label{constant2M}
\eqn
}Hence, a similar analysis can be performed as in the previous subsection provided that $ \varepsilon < (1-\gamma)\, \min \{ 1, \text{dist}(\Gamma, z) \} $ and 
\beqn
    \frac{k \,  \varepsilon}{2} \leq  \min\left\{1, \sqrt{ \frac{p_{000}^2 k_1^2 }{4 C_{2,M}^2} + \frac{p_{000} (1 - \delta)}{C_{2,M}} } -  \frac{p_{000} |k_{1,M}| }{2 C_{2,M}} \right\}
    \label{bounddeltaM}
\eqn
for some $ 0 < \delta < 1$, where 
\beqn
k_{1,M}=  \left(\sum_{|m|+|n|=1} (\text{sgn}(n))^{n}  (\text{sgn}(m))^{m}  h_{nm00} +  \frac{2 h_{0010} }{k} \right) \bigg/p_{000} \,.
\label{k1M}
\eqn
Then, by a Taylor expansion, we readily see that
\beqnx
\left| 1 - |{\Re}\langle w_1,w_2\rangle|^2 - \left( \frac{ 2  \left(p_{1\,0\,0} - p_{-1\,0\,0}  - \frac{2 p_{0\,1\,0}}{k}  \right)  }{p_{000}} + k_{1,M} \right)\frac{k \,  \varepsilon}{2}
\right| \leq C \frac{k^2 \,  \varepsilon^2}{4}
\,
\eqnx
with the constant $C_M$ being given by
{\footnotesize
\beqn
C_M &=& \frac{C_{2,M}}{p_{000}}  \left(1 + 2 \left|k_{1,M} \right| \delta^{-1} \right) +  k_{1,M}^2 \delta^{-1} +  \frac{C_{2,M}^2}{p_{000}^2} \delta^{-1}
+ \frac{ 2  \left(|p_{1\,0\,0}| + |p_{-1\,0\,0}|  - \frac{2 |p_{0\,1\,0}|}{k}  \right)  ) | k_{1,M}| }{p_{000}} \notag \\[2mm]
& &
+ \frac{  \left(|p_{1\,0\,0}| + |p_{-1\,0\,0}|  - \frac{2 |p_{0\,1\,0}|}{k}  \right) ^2 + 2 \left(p_{000} +  \left(|p_{1\,0\,0}| + |p_{-1\,0\,0}|  - \frac{2 |p_{0\,1\,0}|}{k}  \right)     \right) C_{1,M} + C_{1,M}^2}{p_{000}^2}, \quad 
\label{constant3M}
\eqn
}when $0<\delta<1$.
Hence, condition \eqref{alpha1} is fulfilled if 
\beqn
\frac{4 (1- \alpha)}{(2 + \alpha)^2}  \geq \left( \frac{ 2  \left(p_{1\,0\,0} - p_{-1\,0\,0}  - \frac{2 p_{0\,1\,0}}{k}  \right)  }{p_{000}} + k_{1,M} \right) \frac{k \,  \varepsilon}{2} + C_M \left( \frac{k \,  \varepsilon}{2}\right)^{2} \,.
\eqn
Likewise, adding the constraint $k \varepsilon < 1/4$, this directly infers the following theorem.
\begin{Theorem}
For $M > 2$, consider a point $z$ such that $0 <R_0 \leq \text{dist}(z,\Gamma) \leq R_1$ for some $R_0, R_1 \in \mathbb{R}$, where $\Gamma = \Gamma_s = \Gamma_r$,  and that $ \frac{k \,  \varepsilon}{2}$ satisfies \eqref{bounddeltaM} for some $0< \delta <1$ and $k \varepsilon < 1/4$,  $ \varepsilon < (1-\gamma) \, \min \left\{ 1,  \text{dist}(\Gamma, z) \right \} $ for some $0 < \gamma <1 $, where $p_{nrs}$ are defined as in \eqref{pnrs}, $h_{nmrs}$ as in \eqref{hnmrs}, and let $C_{1,M},C_{2,M},C_{M}, k_{1,M}$ be as  in \eqref{constant1M}, \eqref{constant2M} \eqref{constant3M}, and \eqref{k1M}, respectively.
The two point scatterers with equal weights located at $z$ and $z+ \varepsilon v$ are $\alpha$-distinguishable with respect to $\Gamma$ if the following inequality holds:
\beqn
 \frac{k \,  \varepsilon}{2}
\leq
\ \left( \frac{ 2  \left(p_{1\,0\,0} - p_{-1\,0\,0}  - \frac{2 p_{0\,1\,0}}{k}  \right)  }{p_{000}} + k_{1,M} \right)  - \sqrt{  \left( \frac{ 2  \left(p_{1\,0\,0} - p_{-1\,0\,0}  - \frac{2 p_{0\,1\,0}}{k}  \right)  }{p_{000}} + k_{1,M} \right)^2 +  \frac{16 (1- \alpha)}{ C_{M} (2 + \alpha)^2}  }   \,.
\label{conclusionM}
\eqn
\end{Theorem}

In order to make the conditions stated in the above theorem practical, we assume that $\delta = 1/2$ and provide, by considering \eqref{series1} and  \eqref{series2},  approximations for the coefficients with
{\footnotesize
\beqnx
C_{1,M} &\approx& p_{000}
\, , \quad
C_{2,M} \approx 2 p_{000} \, , 
\quad \\[2mm]
C_M &\approx &
13  + 8\left|k_{1,M} \right| +  2 k_{1,M}^2 
+ \frac{  \left(|p_{1\,0\,0}| + |p_{-1\,0\,0}|  - \frac{2 |p_{0\,1\,0}|}{k}  \right) ^2 + 2 \left(|p_{1\,0\,0}| + |p_{-1\,0\,0}|  - \frac{2 |p_{0\,1\,0}|}{k}  \right) ( p_{000} + |k_{1,M} | ) }{p_{000}^2}. \quad 
\eqnx
}  With the above approximations of the constants $C_{1,M},C_{2,M},$ and $C_M$ and estimate \eqref{conclusionM}, we obtain an \textbf{anisotropic} optimal density at any point $z$ and at any direction $v$ for contrast reconstruction from the near-field data in terms of $\alpha$-distinguishably as follows:
{\small
\beqn
 h_{z,v,M}(\alpha) &:= & \frac{2}{k}\min \Bigg\{ \frac{1}{4}, \sqrt{ \frac{k_{1,M}^2 }{16} + \frac{1}{4} } -  \frac{|k_{1,M}| }{4} ,
  \left( \frac{ 2  \left(p_{1\,0\,0} - p_{-1\,0\,0}  - \frac{2 p_{0\,1\,0}}{k}  \right)  }{p_{000}} + k_{1,M} \right)  \notag \\[3mm] 
   &&    - \sqrt{  \left( \frac{ 2  \left(p_{1\,0\,0} - p_{-1\,0\,0}  - \frac{2 p_{0\,1\,0}}{k}  \right)  }{p_{000}} + k_{1,M} \right)^2 +  \frac{16 (1- \alpha)}{ C_{M} (2 + \alpha)^2}  }  
    \Bigg\}. \label{finalfunctionM} 
\eqn
}The remarks in the previous subsection are also in order in this subsection.

\subsubsection{Combining the optimal mesh size function along two opposite directions}

Practically, one may combine the optimal mesh-size for $v$ and $-v$ since they are refined in the same direction, and we may consider the following function instead,
$$\widetilde{h_{z,v}} (\alpha) := \min\big\{ h_{z,v}(\alpha) , h_{z,-v}(\alpha) \big\} \, ,$$
as well as
$$\widetilde{h_{z,v,M}} (\alpha) := \min \big\{ h_{z,v,M}(\alpha) , h_{z,-v,M}(\alpha) \big\} \,. $$
This is to forget the information about separating a point forward or backward along the direction $v$.

\section{Adaptive refinement using the optimal mesh function} \label{sec:refine} 

We can apply the function $\widetilde{h_{z,v}}(\alpha)$ to refine a mesh at a point $z$ along a direction $v$ (which we may refer to as a polarization direction). 
Suppose we have a finite number of inhomogeneous inclusions  inside $\Omega$, and we implement the following procedures to recover all  the inhomogeneities:
\begin{enumerate}
\item[(i)] Select a coarse mesh with mesh size $h = \lambda/2 $; 
\item[(ii)] Locate where the reconstruction has an elongated shape at a sampling point $z$, say, a rod with direction $v$;
\item[(iii)] Refine the mesh $\widetilde{h_{z,v}}(\alpha)$ at the point $z$ and the polarization direction $v$ for further reconstruction;  
\item[(iv)] Repeat step (iii)  until all inclusions are separated. 
\end{enumerate}
It is worth emphasizing that the purpose of steps (ii) and (iii) is to determine whether the inhomogeneities are actually two objects or only one long rod.

\section{Numerical illustrations} \label{sec:numerics}

\subsection{Behavior of the mesh-size function $\widetilde{h_{z,v}}(\alpha)$}

In order to understand better the behavior of the optimal mesh-size function, let us numerically illustrate some of the behaviors of $\widetilde{h_{z,v}}(\alpha)$ defined as in \eqref{finalfunction} (see Theorem \ref{important_theorem}).
Assume that the sources $x_s$ and receivers $x_r$ are located at $\Gamma= \{ ( \cos( 4 \pi n/14) , 4 \sin( 2 \pi n/14)  )\} $ for $n=0,1,\cdots, 7$, as shown in Figure \ref{fig:gamma}(a), and let $\Omega = B_4(0)$.

\begin{figurehere}
 \hfill{}\includegraphics[clip,width=0.4\textwidth]{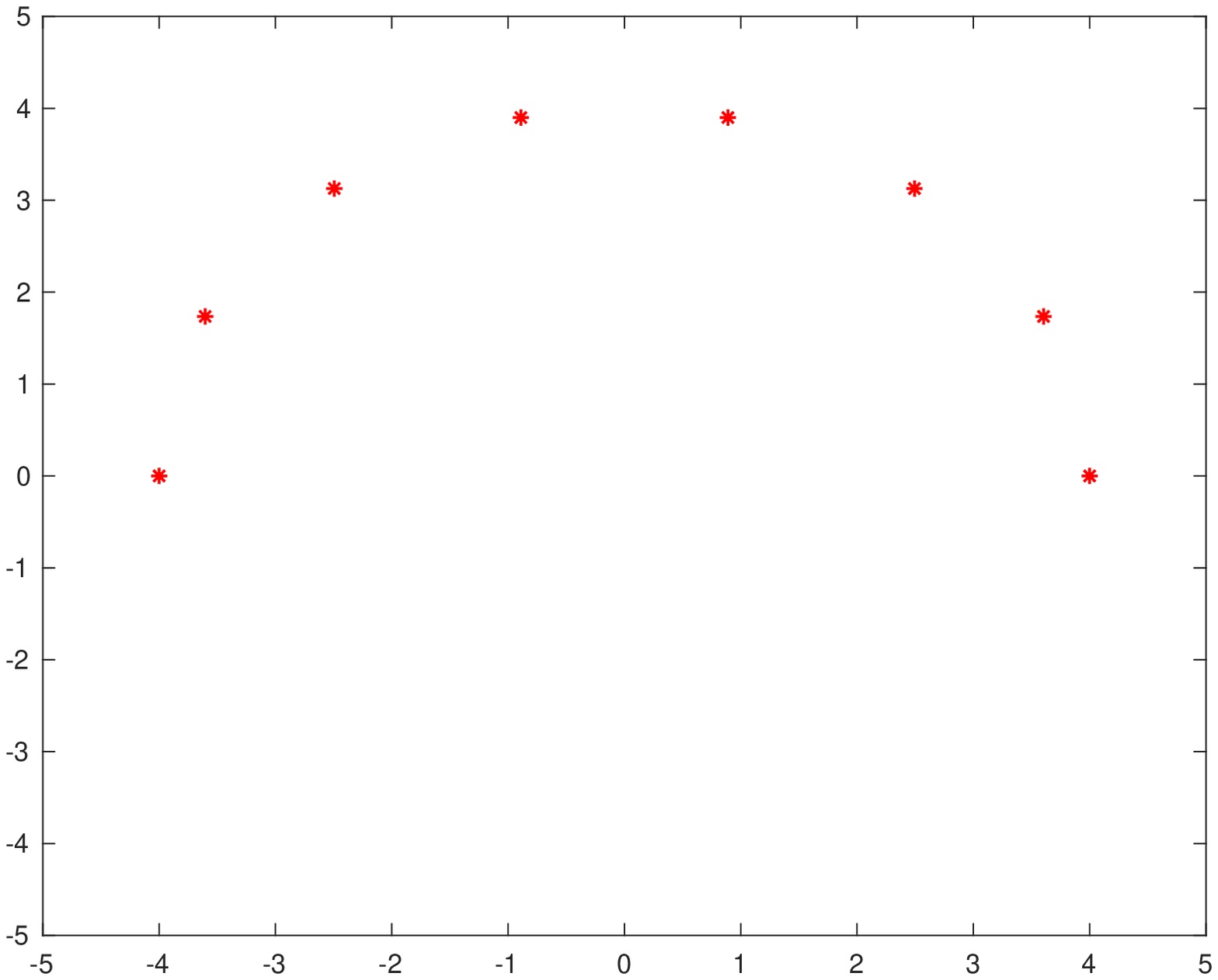}\hfill{}
\hfill{}\includegraphics[clip,width=0.4\textwidth]{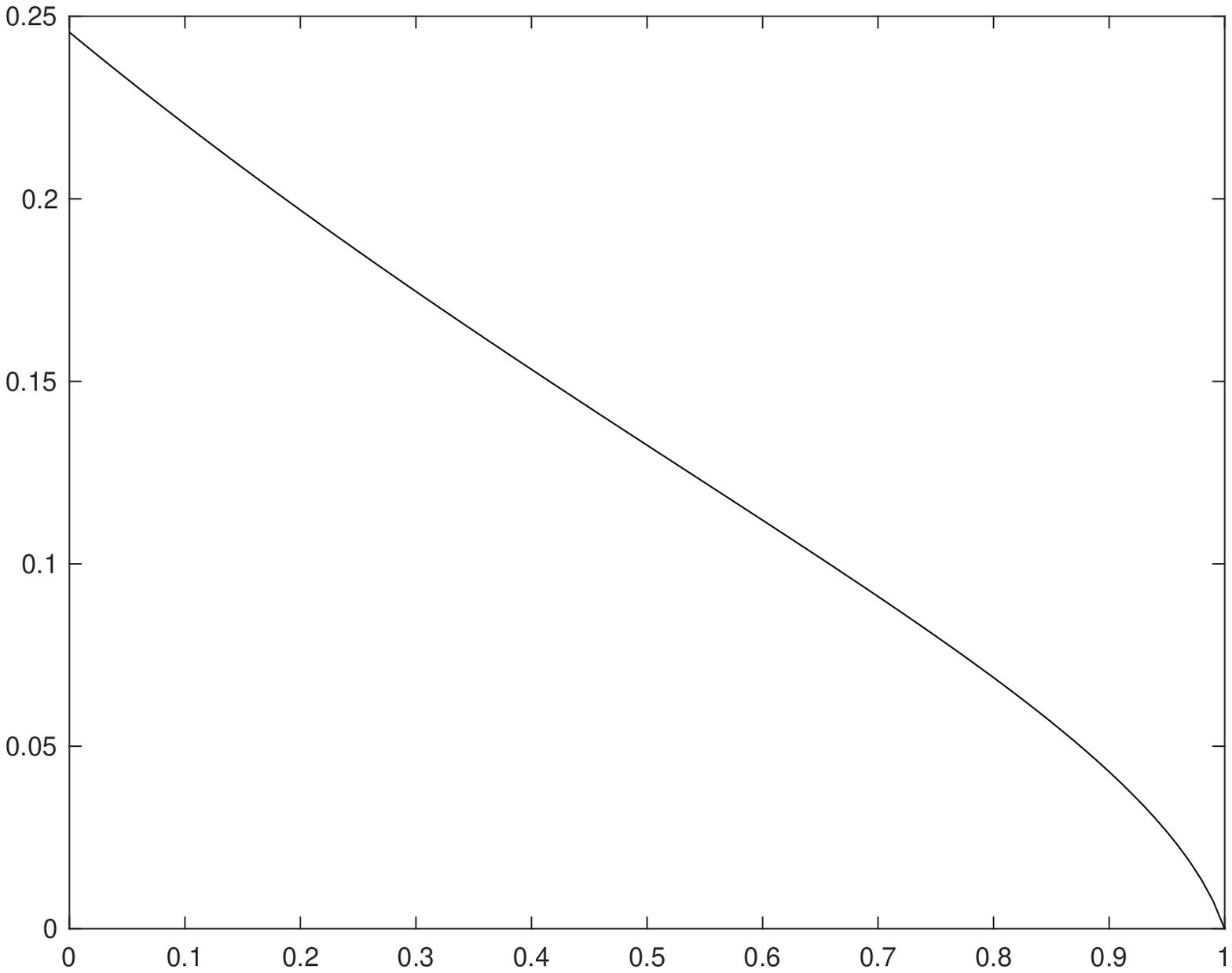}\hfill{}

  \hfill{}(a)\hfill{} \hfill{}(b)\hfill{}
  
 \caption{\label{fig:gamma}  \small{\emph{(a) $\Gamma$ in red. (b) The graph of $h_{z,v}(\alpha)$ with respect to $\alpha$ when $k = 2$, $z = (1,3.5)$ and the polarization direction $v = (1,0)$.  Measurement points in red and grid points in blue. }}}
 \end{figurehere}

We assume that $k = 2$, $z = (1,3.5)$, and the polarization direction $v = (1,0)$. The mesh-size function $h_{z,v}(\alpha)$ with respect to $\alpha$ is illustrated in Figure \ref{fig:gamma}(b).
We may also fix a large $\alpha$ to obtain a density function and uneven mesh with a mesh-size satisfying the density provided by $$ z \mapsto p_{v,\alpha} (z) :=  \frac{ \big(\widetilde{h_{z,v}}(\alpha)\big)^{-2} } { \int_{\Omega}   \big(\widetilde{h_{z,v}}(\alpha) \big)^{-2} dz}$$ for a fixed $v$.
The reason to take this function is that the optimal density is proportional to $ \big( \widetilde{h_{z,v}}(\alpha) \big)^{-M}$ where $M = 2$. Figure \ref{fig:meshsize2} shows the density function $ p_{v,\alpha} (z) $ at $v = (1,0)$ and $v=(0,1)$, respectively,  when $\alpha = 0.995$.

\begin{figurehere}
 \hfill{}\includegraphics[clip,width=0.45\textwidth]{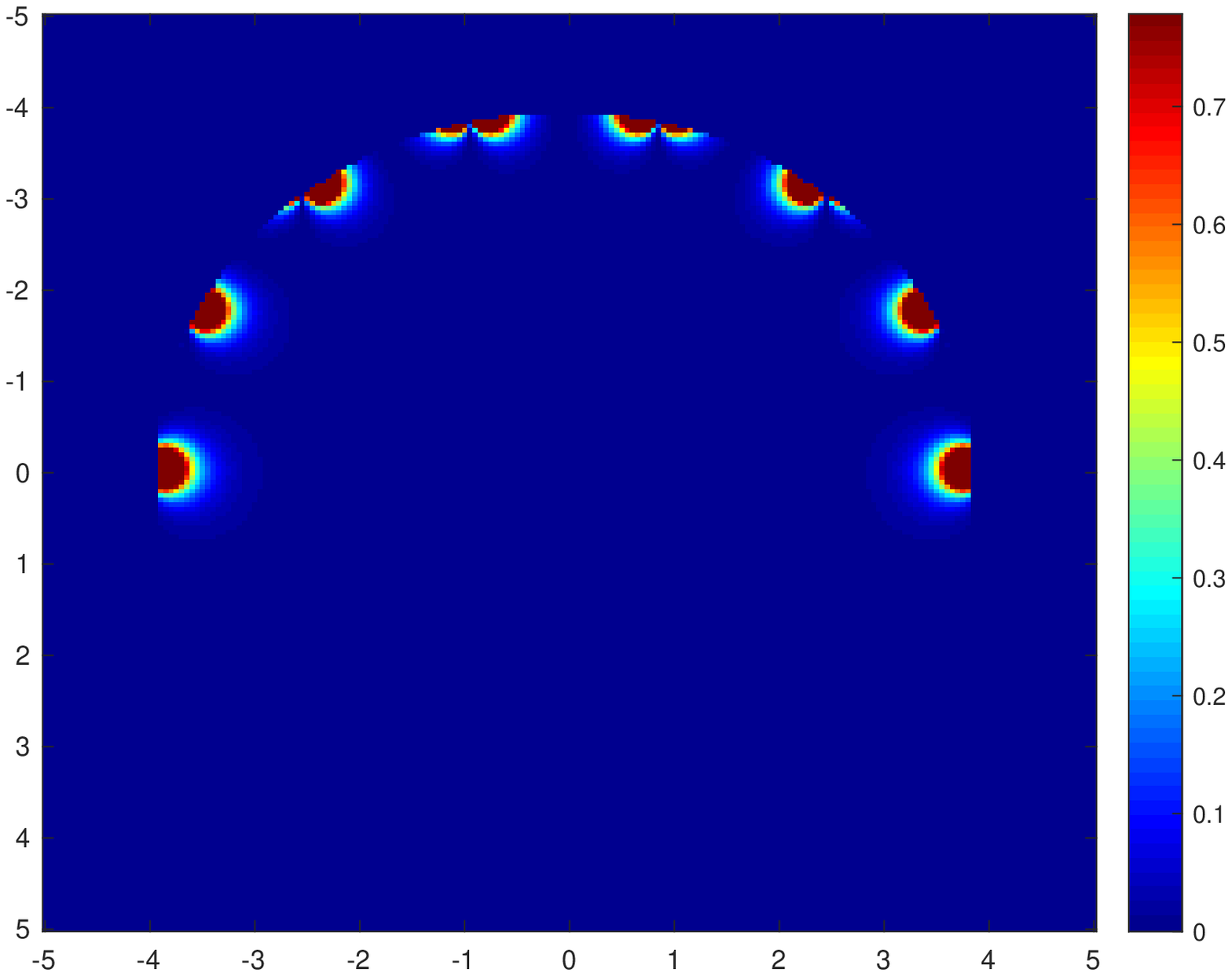}\hfill{}
 \hfill{}\includegraphics[clip,width=0.45\textwidth]{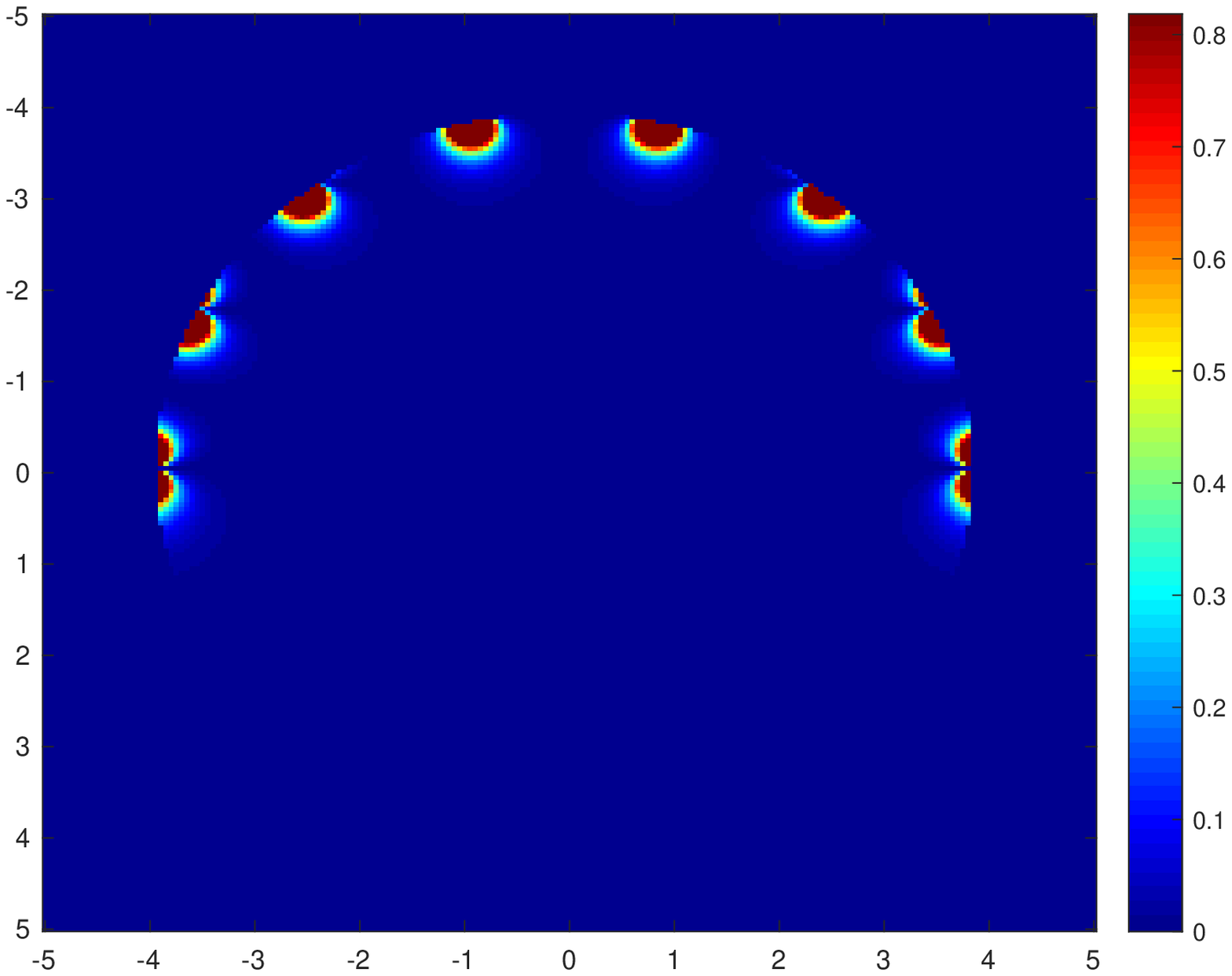}\hfill{}
 \vskip -0.5truecm
 \caption{\label{fig:meshsize2}  \small{\emph{The function $ p_{v,\alpha} (z) $ at $v = (1,0)$ and $v=(0,1)$, respectively, when $\alpha = 0.99$. Measurement points in red and grid points in blue.  }}}
 \end{figurehere}

The above plots show that the refinement is unnecessary when the sampling points are at a distance half the radius from  the sources/receivers.
Now, we fix a particular $\alpha$ and generate a set of grid points with a mesh-size satisfying $h_{z,v}$ along the direction $v$.  
Figure \ref{fig:meshsize3} illustrates the grid points with $v = (1,0)$ and $v=(0,1)$,  respectively,  when $\alpha = 0.55$.

\begin{figurehere}
 \hfill{}\includegraphics[clip,width=0.45\textwidth]{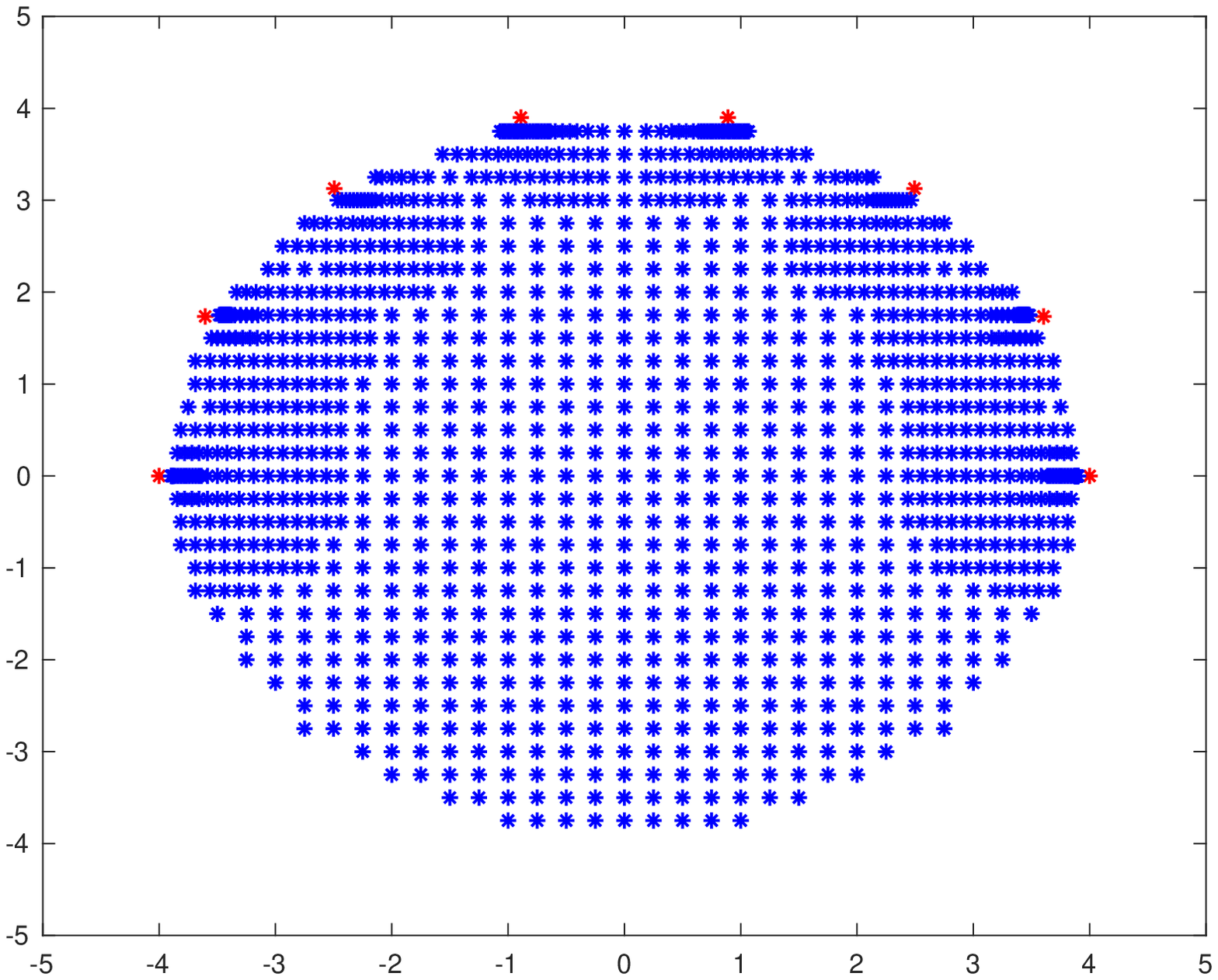}\hfill{}
\hfill{}\includegraphics[clip,width=0.45\textwidth]{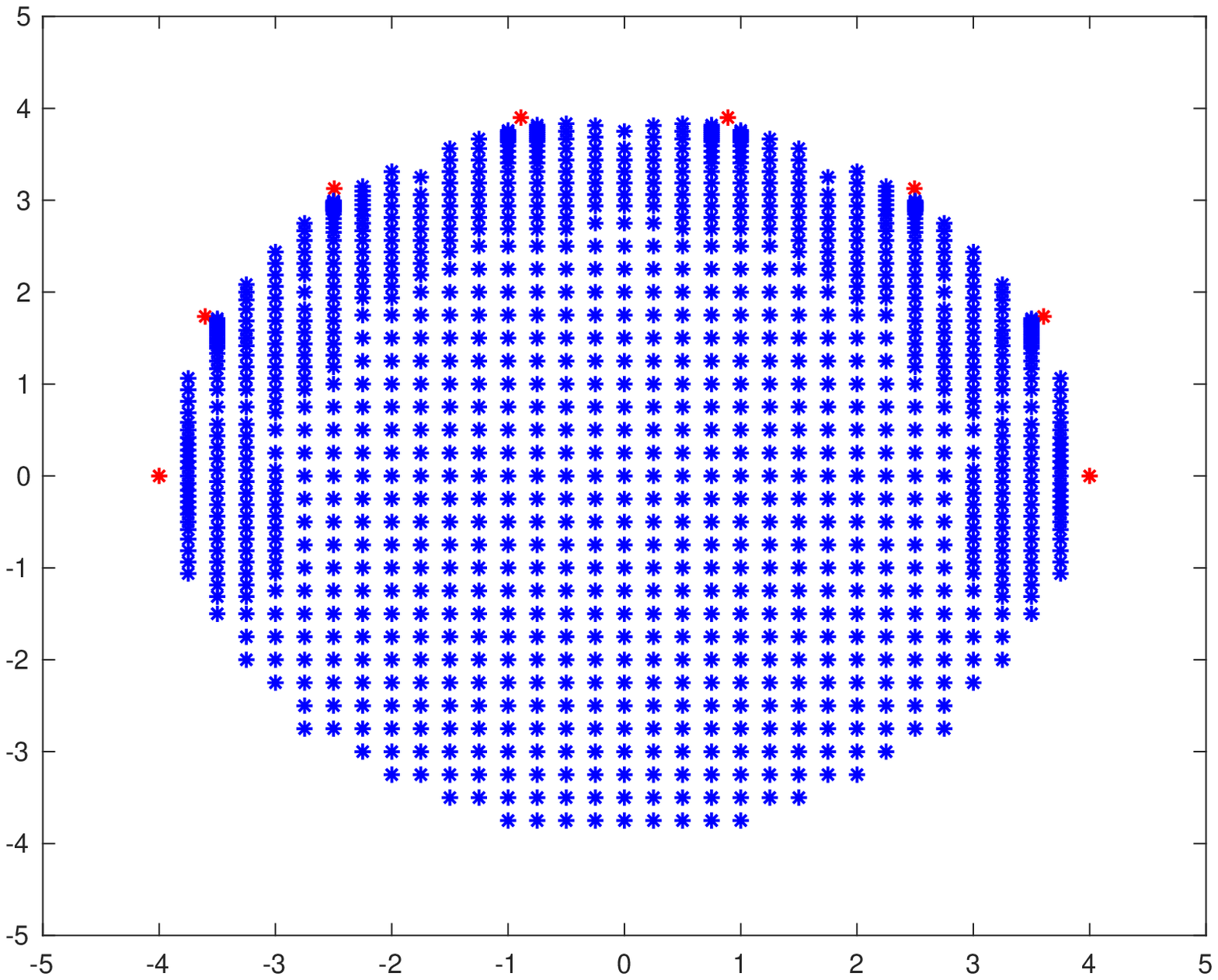}\hfill{}
 \vskip -0.5truecm
 \caption{\label{fig:meshsize3}  \small{\emph{Grid points with mesh-size satisfying $h_{z,v}(0.1)$ along the direction $v$, where $v = (1,0)$ and $v=(0,1)$, respectively. }}}
 \end{figurehere}
 
One may observe that the refinements are done in the directions $v$ near the source/receiver points and are left coarser when the sampling points are further away.
In practice, we may combine the two mesh sizes and generate a tensor mesh with a mesh-size $\widetilde{h_{z,e_1}}(\alpha) $ along  $e_1 = (1,0)$ and $\widetilde{h_{z,e_2}}(\alpha) $ along $e_2 = (0,1)$.
Figure \ref{fig:meshsize4} respectively shows the grid points of the tensor mesh with  $\widetilde{h_{z,e_1}}(\alpha) $ along  $e_1 = (1,0)$ and $\widetilde{h_{z,e_2}}(\alpha) $ along $e_2 = (0,1)$ when $\alpha = 0.55$.

\begin{figurehere}
 \hfill{}\includegraphics[clip,width=0.45\textwidth]{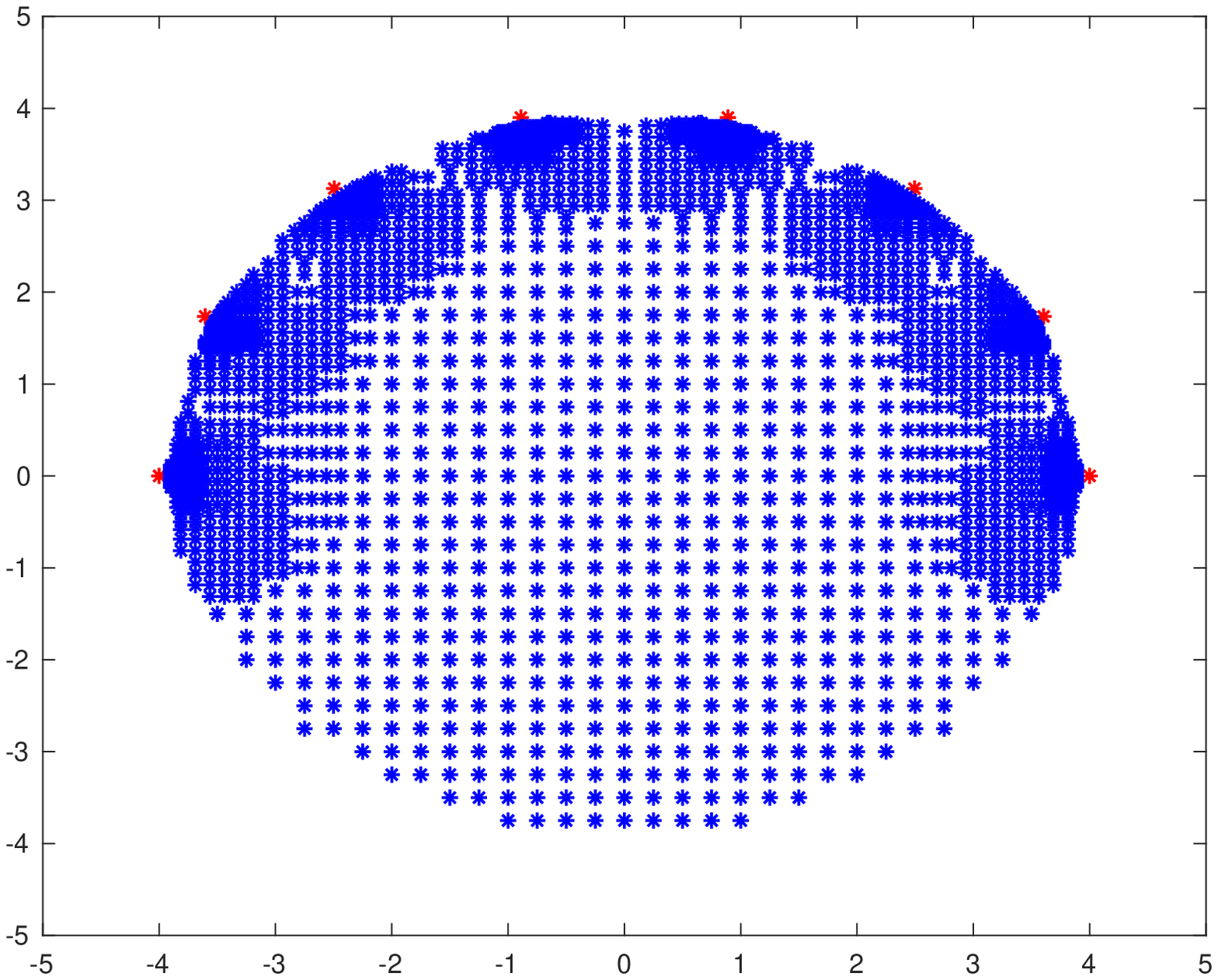}\hfill{}
\hfill{}\includegraphics[clip,width=0.45\textwidth]{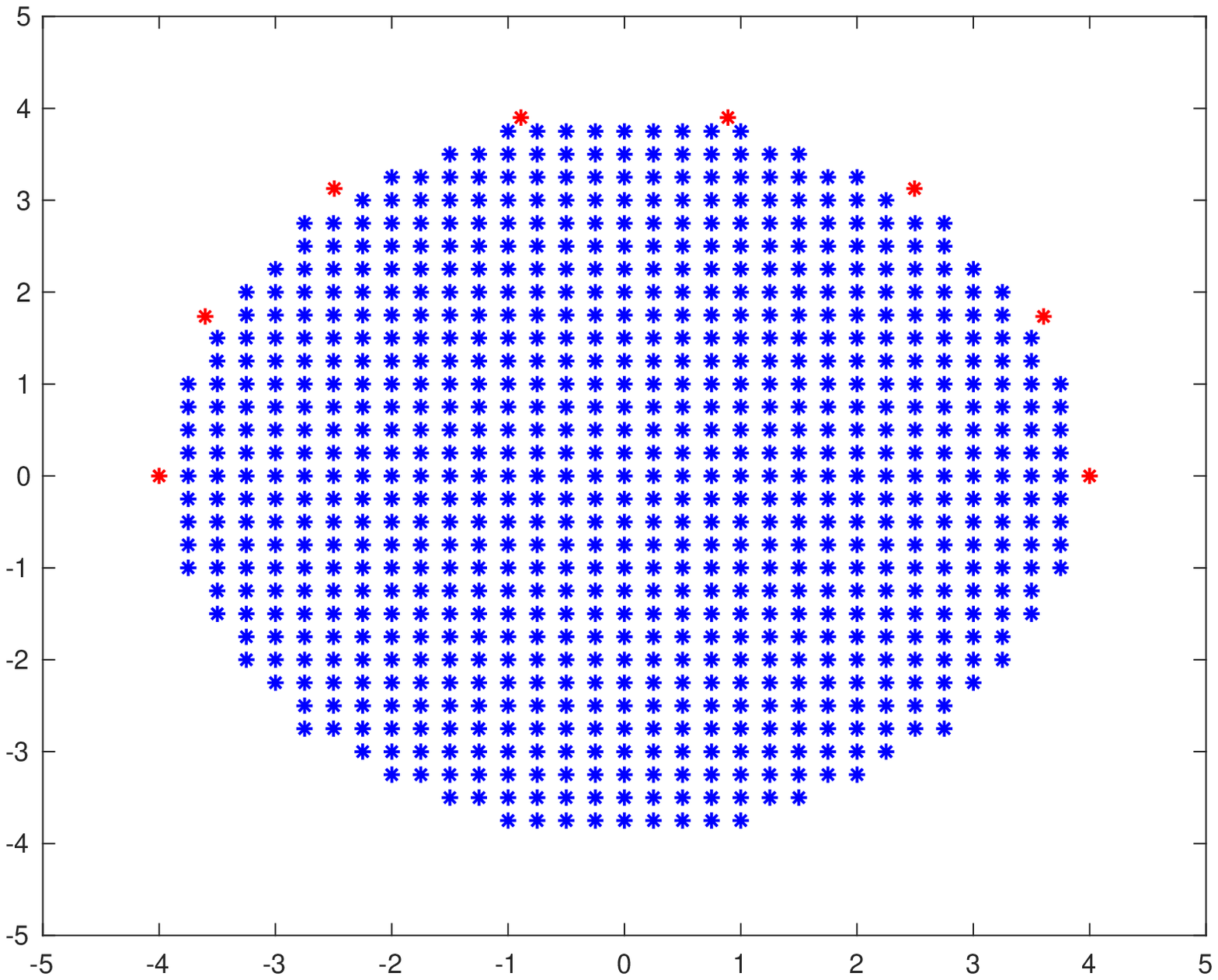}\hfill{}

  \hfill{}(a)\hfill{} \hfill{}(b)\hfill{}
 \caption{\label{fig:meshsize4}  \small{\emph{(a) Grid points with mesh-size satisfying $h_{z,v}(0.1)$ along the direction $v$, where $v = (1,0)$ and $v=(0,1)$, respectively. (b) Unrefined mesh.  Measurement points in red and grid points in blue. }}}
 \end{figurehere}

\subsection{Use of the mesh-size function $\widetilde{h_{z,v}}(\alpha)$ for improving the reconstruction methods}

This paper aims at getting an optimal mesh-size which maximizes the resolution of the reconstruction procedure with minimal numerical cost.  As discussed before, this choice of mesh is not restricted to one particular numerical method, and the numerical methods given in subsection \ref{example_haha} are just several examples to illustrate the concept.   In this subsection, we illustrate  the efficiency of the optimal meshing procedure using the DSM, which is introduced in subsection \ref{example_haha}.  For further details on the DSM, we refer to \cite{L1, L2, LZ,CILZ, L2, LXZ2, LXZ3, DSM11, DSM12, DSM13, DSM14, DSM15}.

In what follows, we consider an example of a near-field measurement corresponding to only one illumination.  Notice that since the case considered here is more pathological than the assumptions given in Theorem \ref{important_theorem}, the mesh-size function $\widetilde{h_{z,v}}(\alpha)$ defined as in \eqref{finalfunction} shall serve as a bottom line for the reconstruction. 

We consider the case  $k = \pi^2 $ when $q(x)$ in \eqref{scattering1} is given by
\beqnx
q(x) = 1 + 2 \mathcal{I}_{Q_1}(x) + 2 \mathcal{I}_{Q_2}(x),
\eqnx
where $Q_1 = \{ x \in \mathbb{R}^2 ; -0.275<x_1<-0.175, -0.275<x_2<-0.175  \} $ and
$Q_2 = \{ x \in \mathbb{R}^2 ; -0.15<x_1<-0.05, -0.15<x_2<-0.05  \} $
are two squares  inside $\mathbb{R}^2$, and $ \mathcal{I}_A$ is an indicator function of a set $A$ such that the value is $1$ if $x\in A$ and is $0$ otherwise.  We take a plane wave incidence with angle $\theta_y = \pi/4$, i.e., $d_y= (\cos(\theta_y), \sin(\theta_y))$.  As shown in the following figures, 
$20$ measurement points along a circle of radius $0.5$ are employed in this example.
We perform the DSM with different mesh sizes to illustrate the efficiency of using $\widetilde{h_{z,v}}(\alpha)$ for meshing.
Note that we intentionally define $Q_1$ and $Q_2$ closer to the boundary measurement points, otherwise the density given by $\widetilde{h_{z,v}}(\alpha)$ will be similar to the situation when we have far-field measurements, and that will not serve the purpose of illustrating the efficiency of $\widetilde{h_{x,v}}(\alpha)$.  

Moreover, we also intentionally put the distance between $Q_1$ and $Q_2$ smaller than half of the wavelength which is approximately $0.16$.  This helps to test if the meshing given by $\widetilde{h_{x,v}}(\alpha)$ can separate the two obstacles.
One remark is that, since we are taking near-field measurements, the DSM would have its kernal possessing ``heavy tail", (see, for instance, \cite{L} for more details), and hence the reconstruction will not be quite sharp and are expected to be more fussy.  We do not anticipate clear reconstructions since only one single measurement is applied in the recovery procedure.
This section, however, illustrates that, taking a mesh-size given by $\widetilde{h_{x,v}}(\alpha)$ helps to obtain as much details as those we can recover from our measurements.

\begin{figurehere}
 \hfill{}\includegraphics[clip,width=0.45\textwidth]{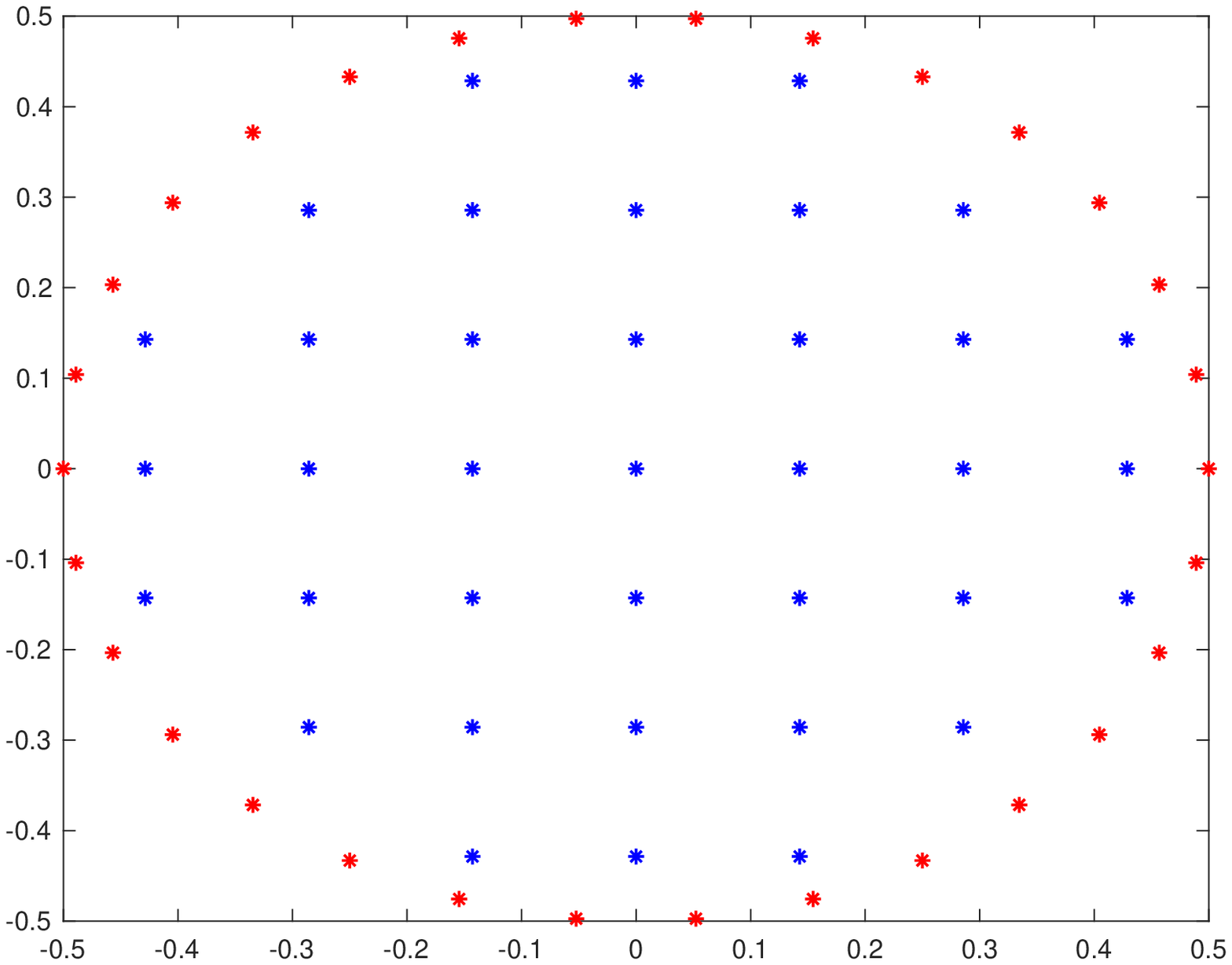}\hfill{}
\hfill{}\includegraphics[clip,width=0.45\textwidth]{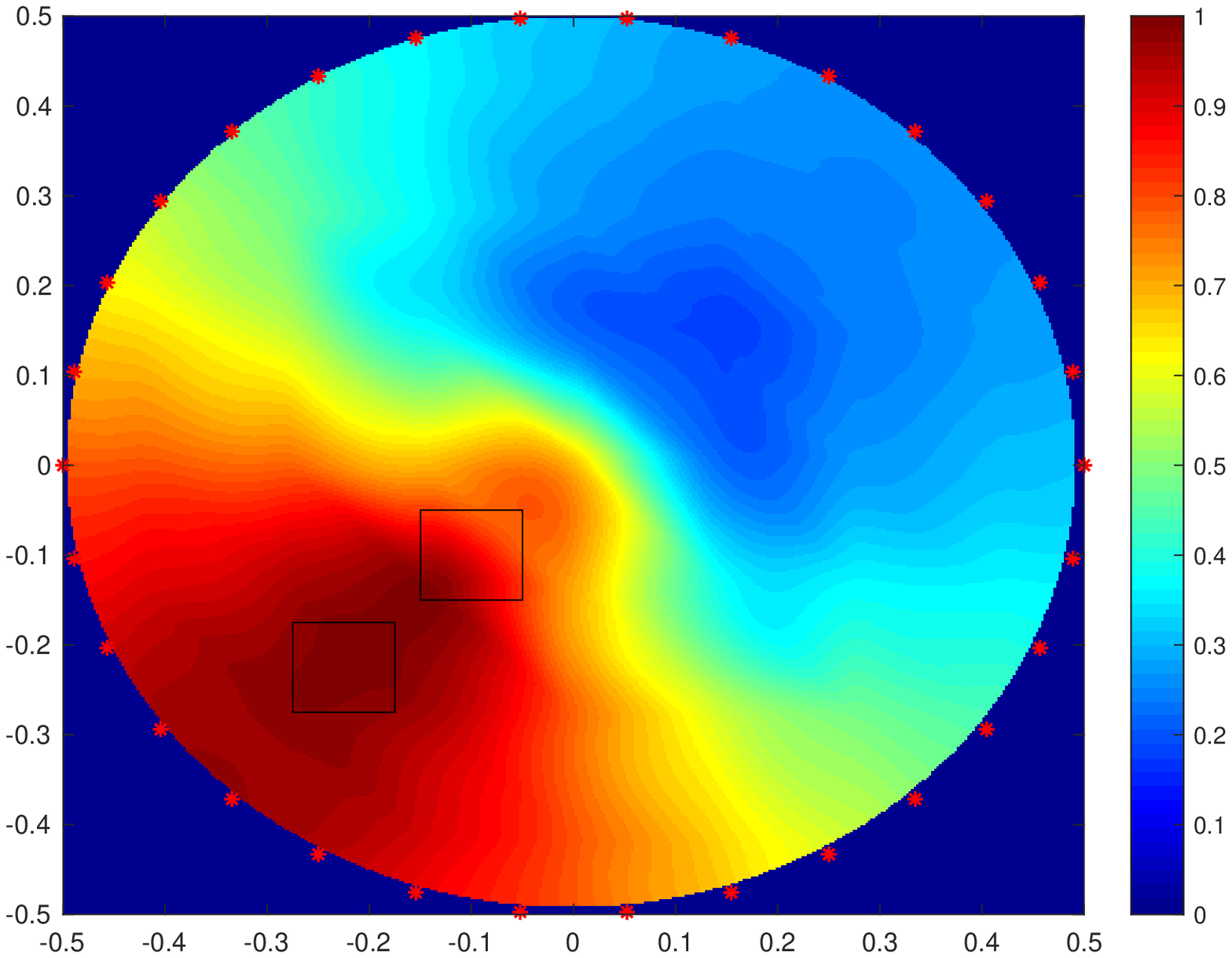}\hfill{}

 \hfill{}(a)\hfill{} \hfill{}(b)\hfill{}
 \vskip -0.2truecm
 \caption{\label{fig:reconstruction}  \small{\emph{(a) Grid points on a coarse mesh.  (b) reconstruction given by $ [I(x,dy) ]^2$, where $I_{\text{coarse}} := I(x,dy)$ in \eqref{eqn:indexfcn} with the grid point given by the coarse mesh. }}}
 \end{figurehere}

Figure \ref{fig:reconstruction} (a) illustrates the coarse mesh that we considered, with its mesh-size being half the wavelength.  
 Figure \ref{fig:reconstruction} (b) shows the reconstruction given by the index function $I_{\text{coarse}} := I(x,dy)$ with this coarse mesh.  As we may expect, we cannot distinguish two inclusions in the image, and basically we only observe a large patch at the lower bottom part of the reconstruction.

Now, let us consider a refinement of the mesh as a tensor grid with mesh-sizes $\widetilde{h_{z,e_i}}(\alpha)$ along the two directions $e_1=(1,0)$ and $e_2=(0,1)$, where we take $\alpha = 0.9$.  We refine the mesh in the whole sampling domain $\Omega$, and obtain a fine mesh given as in Figure \ref{fig:reconstruction2} (a).  We then perform the DMS and Figure \ref{fig:reconstruction2} (b) is the approximation provided by the index function $I(x,dy)$ with this fine sampling.  We notice that this sampling helps to identify whether or not there are indeed two inclusions, with a clear dip separating the two inclusions.  However, we do not expect a very fine reconstruction since the problem is severely ill-posed.  Nonetheless, we are able to separate the two inclusions in the image.

\begin{figurehere}
 \hfill{}\includegraphics[clip,width=0.45\textwidth]{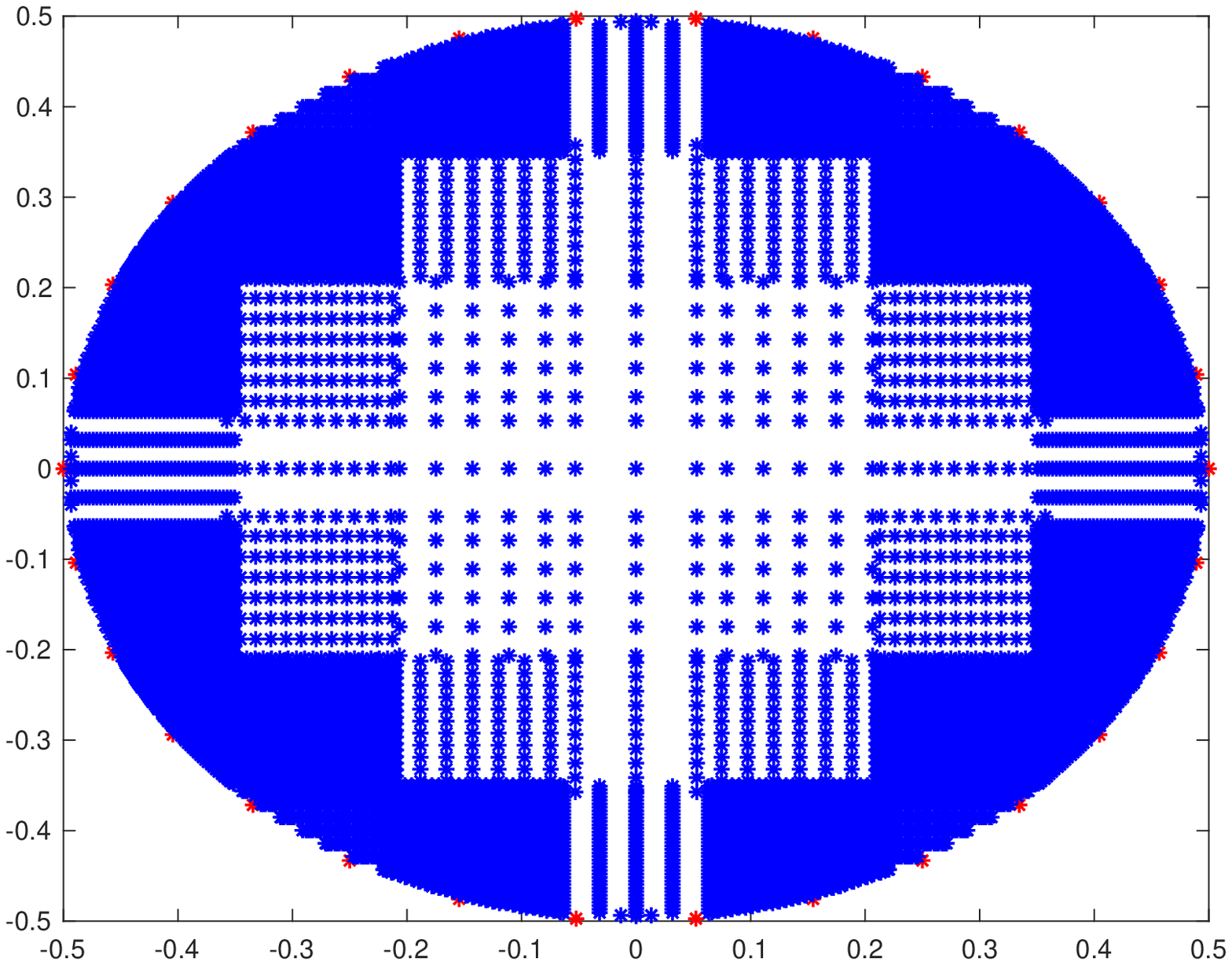}\hfill{}
\hfill{}\includegraphics[clip,width=0.45\textwidth]{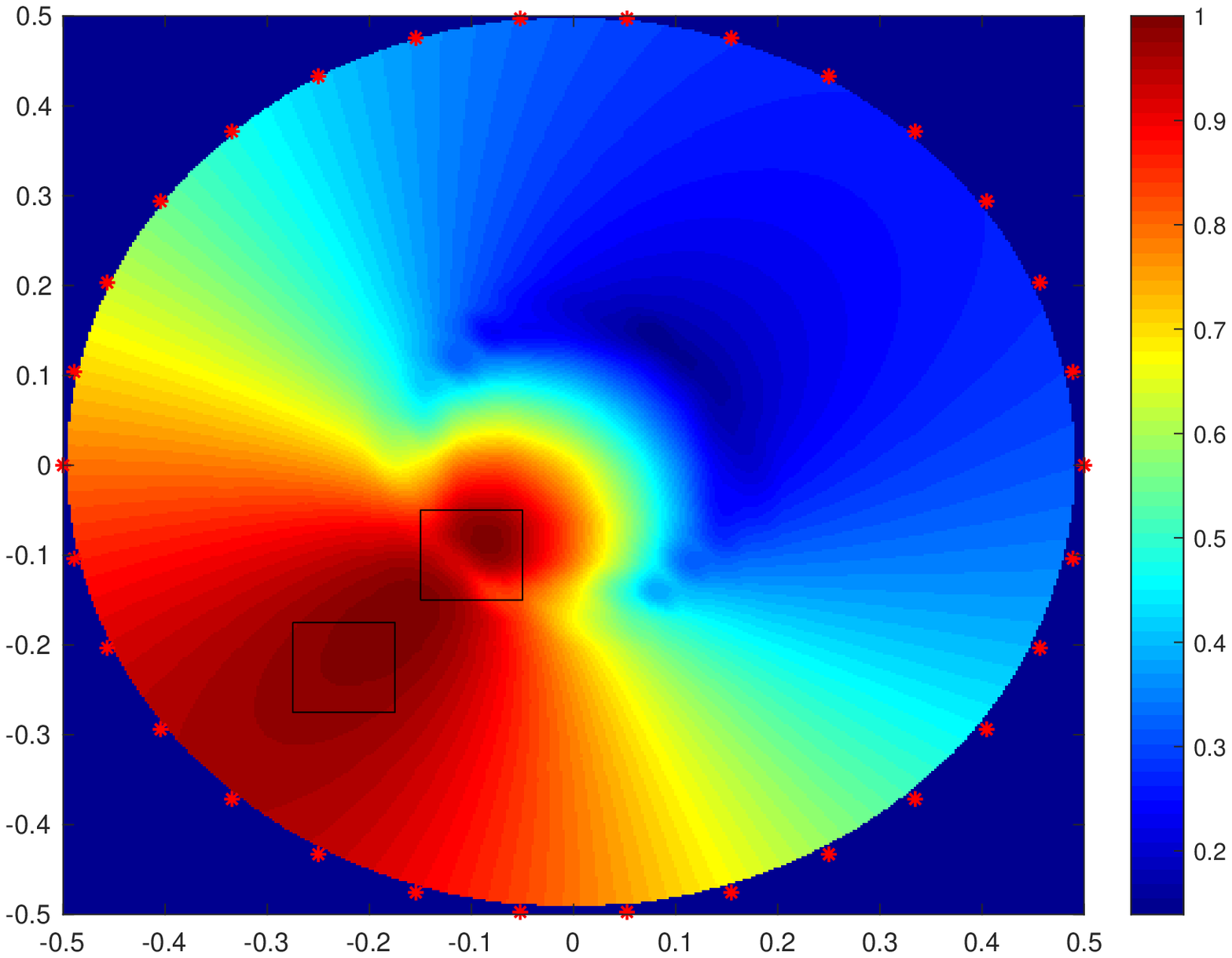}\hfill{}

 \hfill{}(a)\hfill{} \hfill{}(b)\hfill{}
 \vskip -0.2truecm
 \caption{\label{fig:reconstruction2}  \small{\emph{(a) Grid points on a refined mesh given by a tensor grid with mesh-sizes $\widetilde{h_{x,e_i}}(\alpha)$ along $e_1=(1,0), e_2=(0,1)$, $\alpha = 0.9$, and $x$ is taken in the whole sampling domain $\Omega$.  (b) A reconstruction given by $ [I(x,dy) ]^2$, where $I(x,dy)$ in \eqref{eqn:indexfcn} with a refined grid in the left. }}}
 \end{figurehere}

Next, we consider a refinement of the mesh only at the points where the value of $ I_{\text{coarse}} (x) > 0.995 $.
Again, we refine as a tensor grid with mesh-sizes $\widetilde{h_{z,e_i}}(\alpha)$ along the two directions $e_1=(1,0)$ and $e_2=(0,1)$, where we take $\alpha = 0.9$.  The resulting adaptive refinement is shown in Figure \ref{fig:reconstruction3} (a).  We perform the DSM. Figure \ref{fig:reconstruction3} (b) illustrates the recovery obtained by the index function $I(x,dy)$ with this adaptive fine sampling.  Comparing the reconstruction in Figures \ref{fig:reconstruction2} and \ref{fig:reconstruction3}, we can clearly observe  that the qualities of the reconstructions  are similar, i.e., both of them separate the two inclusions, but adaptive refinement requires a significantly smaller number of sampling points in the reconstruction.

\begin{figurehere}
 \hfill{}\includegraphics[clip,width=0.45\textwidth]{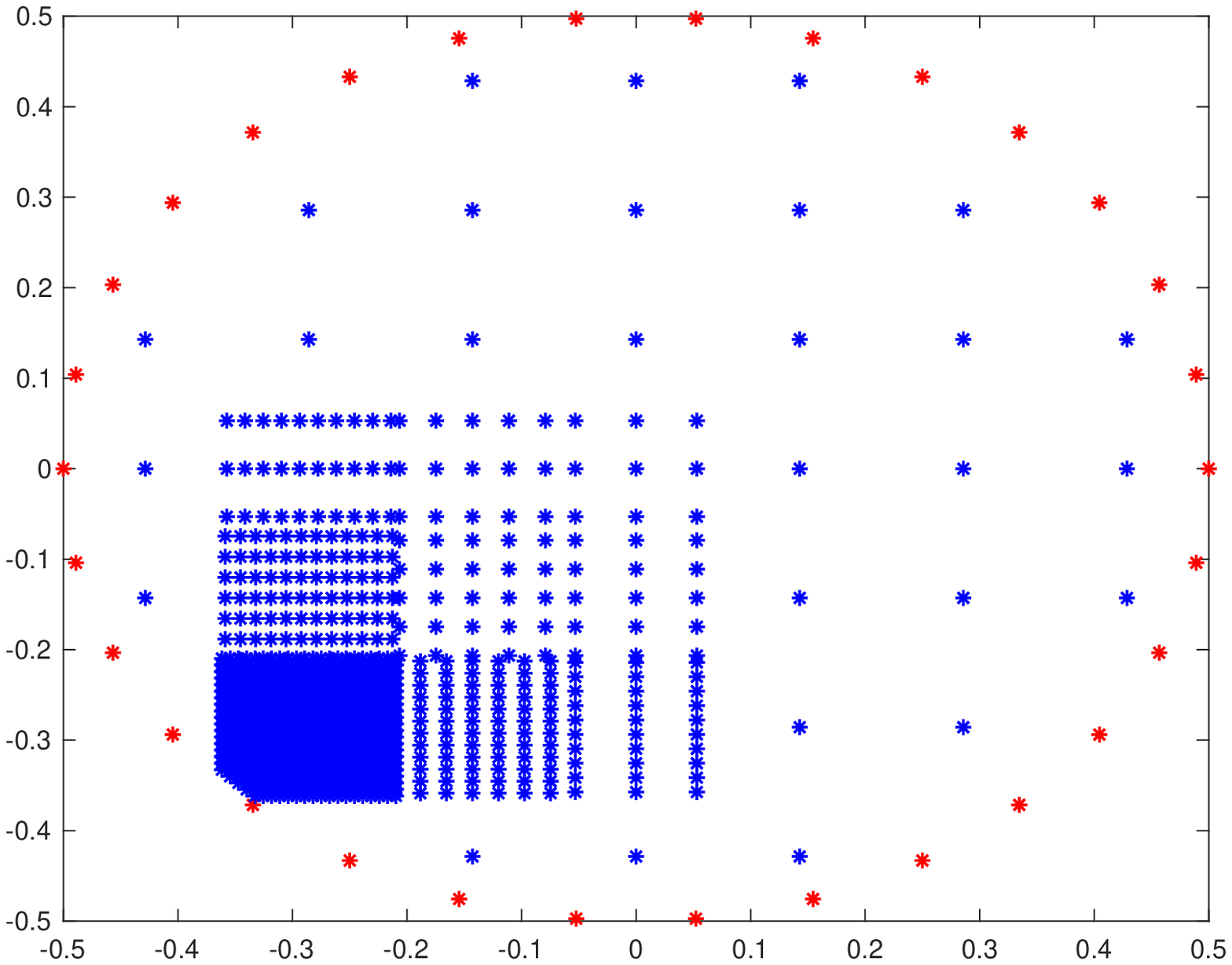}\hfill{}
\hfill{}\includegraphics[clip,width=0.45\textwidth]{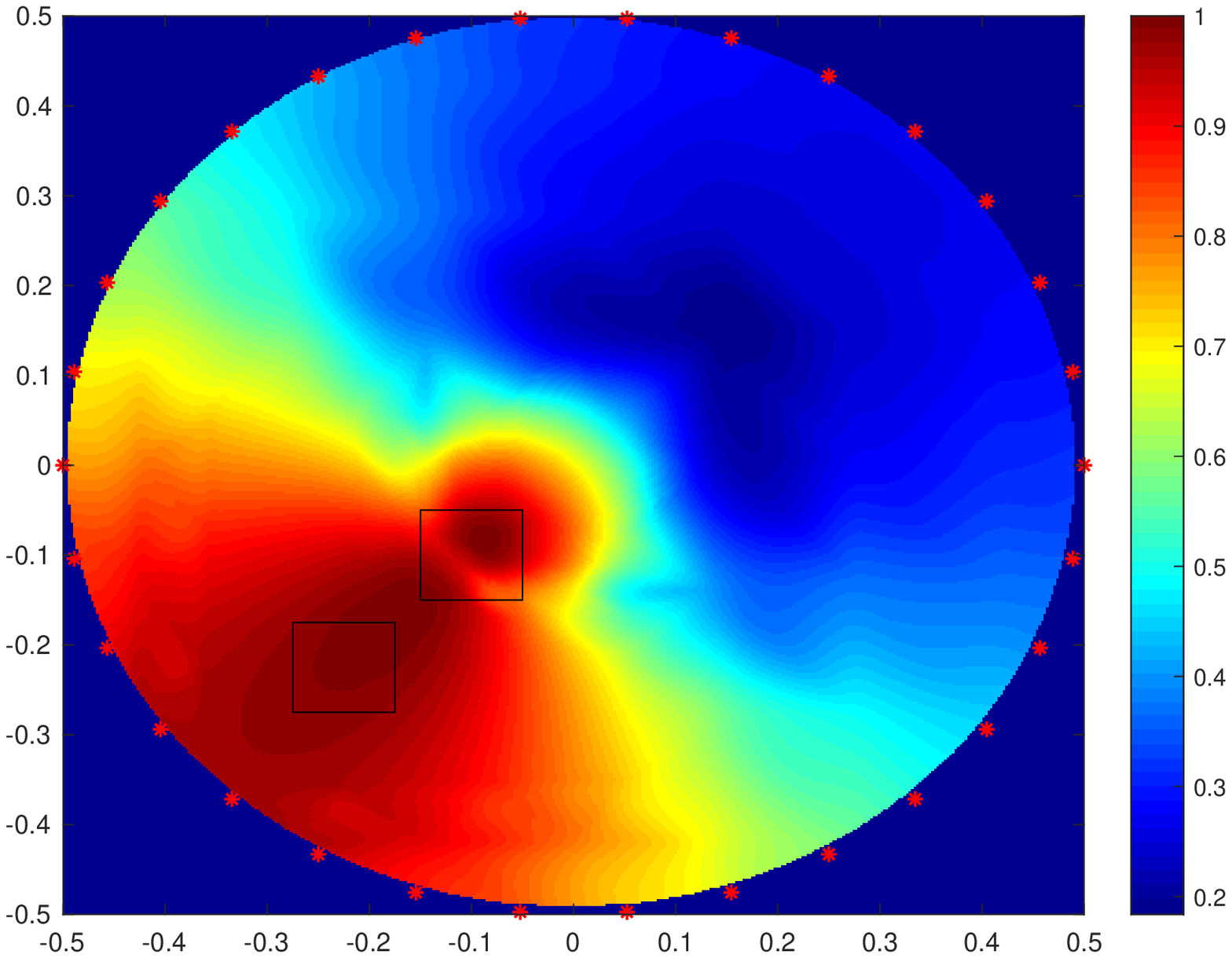}\hfill{}

 \hfill{}(a)\hfill{} \hfill{}(b)\hfill{}
 \vskip -0.2truecm
 \caption{\label{fig:reconstruction3}  \small{\emph{(a) Grid points on adaptively-refined mesh given by a tensor grid with mesh-sizes $\widetilde{h_{x,e_i}}(\alpha)$ along $e_1=(1,0), e_2=(0,1)$, $\alpha = 0.9$, and $x$ is taken only when $I_{\text{coarse}}(x) > 0.995$.  (b) A reconstruction given by $ [I(x,dy) ]^2$, where $I(x,dy)$ in \eqref{eqn:indexfcn} with adaptively-refined grid on the left. }}}
 \end{figurehere}

\section{Conclusion} \label{con}

In this paper, we have discussed the choice of optimal mesh size for the reconstruction of inclusions from both the near-field and the far-field data in two and three dimensions.  The computation complexity would be sharply reduced with an optimal mesh size for state-of-the-art algorithms (i.e., LSM, DSM, CSIM, etc,) since the computation complexity is usually at least of the order of $\mathcal{O}(N^{M-1})$ in $\mathbb{R}^M$, where $N$ is the number of sampling points along one direction.  Moreover, an optimized mesh size can help improving reconstructions of some iterative reconstruction approaches (i.e., MSM, EMSM, etc,) 
by avoiding choices of stopping criterion which may heavily depend on the quality of the reconstruction under a coarser mesh and a subjective choice of the cut-off value.   
We can conclude that our results are expected to have important implications in solving inverse scattering problems.

\section{Acoknowledgement}
The work of Yat Tin Chow was supported by the ONR grant N000141712162 and the NSF grant DMS-1720237. 
The work of Keji Liu was supported by the NNSF of China under grants Nos. 11601308 and 11771349.

\end{document}